\documentclass{amsart}
\usepackage{verbatim}
\usepackage{rotating} 
\usepackage{amssymb}
\usepackage{mathrsfs,amsfonts,amsmath,amssymb,epsfig,amscd,xy,amsthm,pb-diagram} 

\newtheorem{theorem}{Theorem}[section]
\newtheorem{proposition}[theorem]{Proposition}

\newtheorem{lemma}[theorem]{Lemma}

\newtheorem{question}[theorem]{Question}
\newtheorem{definition}[theorem]{Definition}

\theoremstyle{plain}

\theoremstyle{remark}

\newtheorem{remark}[theorem]{Remark}

\newcommand{\C}{{\mathbb C}}

\newcommand{\Q}{{\mathbb Q}}

\newcommand{\Z}{{\mathbb Z}}

\newcommand{\cC}{{\mathcal C}}
\newcommand{\cA}{{\mathcal A}}

\newcommand{\cI}{{\mathcal I}}

\newcommand{\fq}{\mathfrak q}

\newcommand{\fp}{\mathfrak p}

\newcommand{\Gm}{\mathbb{G}_{\text{m}}}
\newcommand{\Gmn}{\mathbb{G}_{\text{m}}^{n}}

\DeclareMathOperator{\Spec}{Spec}

\DeclareMathOperator{\Aut}{Aut}

\newcommand{\bP}{{\mathbb P}}

\newcommand{\bA}{{\mathbb A}}

\newcommand{\cO}{\mathcal{O}}
\newcommand{\cB}{\mathcal{B}}

\newcommand{\scrO}{\mathscr{O}}
\newcommand{\scrX}{\mathscr{X}}
\newcommand{\scrV}{\mathscr{V}}

\newcommand{\scrP}{\mathscr{P}}
\newcommand{\scrG}{\mathscr{G}}
\newcommand{\scrL}{\mathscr{L}}
\newcommand{\scrM}{\mathscr{M}}
\newcommand{\scrZ}{\mathscr{Z}}

\DeclareMathOperator{\textpp}{pp}

\author{Khoa Nguyen}
\address{
Khoa Nguyen \\
Department of Mathematics\\
University of California\\
Berkeley, CA 94720 
}

\email{khoanguyen2511@gmail.com}
\urladdr{www.math.berkeley.edu/\~{}khoa}


\begin{document}
	\title[Diagonally Split Polynomial Maps]{Some Arithmetic Dynamics of Diagonally Split Polynomial Maps}
	\date{8/21/2013 (version 10).}
	\begin{abstract}		 
		Let $n\geq 2$, and let $f$ 
		be a polynomial of degree at least 2 with coefficients in
		a number field or a characteristic 0 function field $K$. We present
		two arithmetic applications of a recent theorem
		of Medvedev-Scanlon to
		the dynamics of the 
		map
		$(f,...,f):\ (\bP^1_{K})^n\longrightarrow (\bP^1_K)^n$, 
		namely
		the dynamical analogues of the Hasse principle and the 
		Bombieri-Masser-Zannier 
		height bound theorem. In particular, we prove that
		the Hasse principle holds when we intersect an orbit and a preperiodic
		subvariety, and that points in the intersection of a 
		curve with the union of all periodic hypersurfaces 
		have bounded heights unless that curve
		is vertical or contained in a periodic 
		hypersurface. 
	\end{abstract}
	\maketitle
	\section{Introduction} \label{common}
	Let $n\geq 2$ be an integer, $K$ a field of characteristic 0,  and $f$ a polynomial of degree 
	$d\geq2$ in $K[X]$. Let $\varphi=(f,\ldots,f)$ be the corresponding self-map of $(\bP^1_K)^n$.
	When $f$ is conjugate to $X^d$ or $\pm C_d(X)$, where $C_d(X)$ is the Chebyshev polynomial of
	degree $d$, the $\varphi$-preperiodic subvarieties of $(\bP^1_K)^n$
	``essentially come from'' the torsion translates of subgroups of 
	the torus $\Gmn$. In a recent work \cite{MedSca}, Medvedev and Scanlon define disintegrated 
	polynomials of degree $d$ to be those that are
	not conjugate to $X^d$ or $\pm C_d(X)$. When $f$ is disintegrated, they can also give
	an explicit description of $\varphi$-periodic subvarieties of $(\bP^1_K)^n$. When $K$ is a number field
	or a function field of a curve, the arithmetic dynamics of $\varphi$ is an interesting
	topic since it is the dynamical analogue of the arithmetic of $\Gmn$ which is an active
	area of research, for examples see \cite{Zan}. This paper presents two applications of the
	Medvedev-Scanlon description to certain (unlikely) intersections involving $\varphi$-preperiodic
	subvarieties of $(\bP^1_K)^n$.
	
	For the rest of this paper, let $K$ be a number field or a function field. Our first application
	is called the (strong) dynamical Hasse principle in \cite{AKNTVV}, as follows. Given
	a projective variety $X$, a self-map $\phi$ of $X$, a closed subvariety $V$, all defined over $K$. 
	Given a $K$-rational point
	$P\in X(K)$ such that the $\phi$-orbit:
	$$\cO_{\phi}(P)\colon=\{P,\phi(P),\ldots\}$$
	does not intersect $V(K)$. Under certain extra conditions, one may ask if there are infinitely many 
	primes $\fp$ of
	$K$ such that the $\fp$-adic closure of $\cO_{\phi}(P)$ does not intersect
	$V(K_\fp)$. This is the same as requiring the orbit of $P$ does not intersect $V$ after taking modulo
	$\fp^m$ for a sufficiently large $m$. This kind of question is first investigated by Hsia and 
	Silverman
	\cite{hsiasil_BM} with motivation from the Brauer-Manin obstruction to the Hasse principle in
	diophantine geometry. We refer the readers to \cite{hsiasil_BM} and the references there for
	more details. As far as we know, all the papers treating the dynamical Hasse principle so far
	either assume that $\dim(V)=0$ \cite{SilVo}, \cite{BGHKST}, 
	or that $\phi$ is \'etale \cite{hsiasil_BM}, \cite{AKNTVV}. By combining results and techniques in
	\cite{SilVo} and \cite{AKNTVV} in addition to the Medvedev-Scanlon theorem, we are able
	to give first examples when $\dim(V)>0$ and $\phi$ is not \'etale:

   \begin{theorem}\label{common1}
   	 Let $f$ be a polynomial of degree at least 2 in $K[X]$,
   	 and $\varphi=(f,...,f):\ (\bP^1_{K})^n\longrightarrow (\bP^1_{K})^n$. Let
   	 $V$ be an absolutely irreducible preperiodic curve or hypersurface 
   	 in $(\bP^1_{K})^n$, and
   	 $P\in (\bP^1)^{n}(K)$ such that the $\varphi$-orbit of $P$ does not
   	 intersect $V(K)$. 
			Then there are infinitely many primes $\fp$ of $K$ such that
			the $\fp$-adic closure of the orbit of $P$ does not intersect $V(K_\fp)$,
			where $K_\fp$ is the $\fp$-adic completion of $K$.   
   \end{theorem}
  
  Our second application should be called the dynamical Bombieri-Masser-Zannier height bound.
  With motivation from the Manin-Mumford conjecture, Lang asks whether a curve in $\Gmn$ that 
  is not a torsion translate of a subgroup has only finitely many torsion points.
  Lang's question is an instance of unlikely intersections as explained in
  \cite{Zan}, and an affirmative answer has been given by Ihara, Serre and Tate independently. 
  In the original paper
  \cite[Theorem 1]{BMZ99}, Bombieri, Masser and Zannier
  proceed further by investigating the question of ``complementary dimensional intersections'', such as 
  the 
  intersection
  of a curve that is not contained in a translate of a subgroup
  with torsion translates of subgroups of codimension one. Recently, a dynamical analogue of
  the Manin-Mumford conjecture and Lang's question has been proposed by Zhang
  \cite{ZhangDist}, and modified by Zhang, Ghioca and Tucker \cite{GTZ}. However, as
  far as we know, a dynamical
  ``complementary dimensional intersection''   
  analogue of the Bombieri-Masser-Zannier theorem has not been treated elsewhere.
  By applying the Medvedev-Scanlon theorem and basic (canonical) height arguments, this paper is the first to establish such a dynamical analogue:
  	\begin{theorem}\label{common2}
  		Let $f$, and $\varphi$ be as in Theorem \ref{common1}. 
  		Assume that $f$ is disintegrated. 
 Let $C$ be an irreducible curve in
  		$(\bP^1_{\bar{K}})^n$ that is not contained in any periodic
  		hypersurface. Assume that  $C$ 
  		maps surjectively 
  		onto each factor $\bP^1$ of $(\bP^1)^n$. Then points in
  		$$\bigcup_{V} (C(\bar{K})\cap V(\bar{K}))$$
  		have bounded heights, where $V$ ranges over all
  		periodic hypersurfaces of $(\bP^1_{\bar{K}})^n$.
  	\end{theorem}

  The above two theorems are examples of the main results and topics treated in this paper. 
  We refer the readers to Theorems \ref{intro_thm3}, \ref{intro_thm1},
  \ref{intro_thm1*}, \ref{dynBMZ_1}, \ref{thm:C_cap_V}, \ref{thm:H_cap_V}, and 
  \ref{thm:finalsplit}
  for much more general results. This paper is the result of a reorganization
  and slight expansion
  of the paper \cite{sister}. We are grateful to an anonymous comment that 
  the function field case might not be necessary for the dynamical Hasse principle. If the 
  constant ground field $\kappa$ of $K$ is uncountable, then for all but countably many
  places $\fp$ of $K$, the orbit $\cO_{\varphi}(P)$ is discrete 
  in the $\fp$-adic topology. It has been suggested to the author that 
  even when $\kappa$ is countable, although it is not completely trivial, it is still likely that
  the above orbit remains discrete in the $\fp$-adic topology for infinitely many primes $\fp$ of $K$.
  Since we do not know a proof of this result, and since, more importantly, the techniques
  in this paper work equally well for the function field case with just little extra effort,
  we decide to keep it.
   
  The organization of this paper is as follows. First, we present the Medvedev-Scanlon description
  of $\varphi$-preperiodic subvarieties of $(\bP^1_K)^n$
  in a way most suitable for our applications. Then we use these results to investigate 
  the dynamical Hasse principle and the dynamical Bombieri-Masser-Zannier height bound theorem. 
  We finish
  this section by stating our convention for notation. 
  A function field means a finitely generated field
  of transcendental degree 1 over a ground 
  field of characteristic 0. Throughout
  this paper, $K$ denotes a number field or a function field over the ground field $\kappa$, and 
  $M_K$
  denotes the set of places of $K$. In the function field case,
  by places of $K$, we mean the equivalence classes of the valuations
  on $K$ that are trivial on $\kappa$. We assume that $\kappa$ is relatively 
  algebraically 
  closed in $K$, or equivalently, $\kappa^{*}$ is exactly the elements of $K^{*}$
  having valuation 0 at every place. This assumption
  will not affect the generality of our results. 
  For every $v$ in $M_K$, 
  let $K_v$ denote the completion of $K$ with respect to $v$. If $v$
  is non-archimedean, we also let $\scrO_v$ and $k_v$ respectively denote
  the valuation ring and the residue field of $K_v$. By a variety over $K$,
  we mean a reduced separated scheme of finite type over $K$. Every Zariski
  closed subset of a variety is identified with the corresponding induced
  reduced closed subscheme structure, and is called a closed subvariety.
  Curves, surfaces,\ldots, and hypersurfaces are not assumed
  to be irreducible but merely \textit{equidimensional}.
  In this paper, $\bP^1_{K}$
  is implicitly equipped with a coordinate function $x$ having
  only one simple pole and zero which are denoted by $\infty$ and $0$ 
  respectively. Every polynomial $f\in K[X]$ gives a corresponding self-map
  of $\bP^1_{K}$, and a self-map of $\bA^1_{K}=\bP^1_{K}-\{\infty\}$ 
  by its action on $x$. For every self-map $\mu$ of a set,
  for every positive integer $n$, we write $\mu^n$ to denote the
  $n^{\text{th}}$ iterate of $\mu$, and we define $\mu^0$ to be
  the identity map. The phrase ``for almost all'' means ``for all but
  finitely many''.
  
  {\bf Acknowledgments.}  The author would especially like to thank Tom Tucker who gave important ideas as well as constant encouragement to the author.
We are grateful to Tom Scanlon for numerous helpful conversations, 
explanations and advice. We would also like to thank Alice Medvedev, Paul Vojta, Xinyi Yuan and Mike Zieve
for helpful conversations. We wish to express our gratitude to
anonymous comments made
to the paper \cite{sister}.
The author was partially supported by NSF Grant
DMS-0901149.
     
  \section{The Medvedev-Scanlon Theorem}\label{MSTheorem}
				Throughout this section, let $F$ be an algebraically closed
				field of characteristic 0, and $n\geq 2$ 
				a positive integer. We now introduce the notion of disintegrated polynomials.
			For $d\geq2$,
				the Chebyshev polynomial of degree $d$ 
				is the unique polynomial 
				$C_d\in F[X]$ such that $\displaystyle 
				C_d(X+\frac{1}{X})=X^d+\frac{1}{X^d}$.
				\begin{definition}\label{disintegrated}
				Let $f\in F[X]$ be a polynomial of 
				degree $d\geq 2$. Then $f$ is said to be special
				if there is $L\in \Aut(\bP^1(F))$ such that
				$L^{-1}\circ f\circ L$ is either $\pm C_d$ or
				the power monomial $X^d$. The polynomial $f$ is
				said to be disintegrated if it is not special.
				 \end{definition}
			Here we have adopted the terminology ``disintegrated polynomials'' used in the Medvedev-Scanlon work \cite{MedSca}
			which has its origin from model theory. 
			Unfortunately, there is no standard terminology for what we call special 
			polynomials. 
			Complex
			dynamists describe such maps as having ``flat orbifold metric'', Milnor 
			\cite{Milnor}
			calls them ``finite quotients of affine maps'', and Silverman's book
			\cite{Sil-ArithDS}
			describes them as polynomials ``associated to algebraic groups''.
			The term ``special'' used here is succinct and sufficient for our purposes.
			We remark that for every $m>0$, $f^m$ is disintegrated
			if and only $f$ is disintegrated. To prove this, we may assume
			$F=\C$ by the Lefschetz principle and use the fact that a polynomial
			 is disintegrated if and only if its Julia set
			 is not an interval or a circle.
				We have
				the following theorem of Medvedev-Scanlon \cite[p. 5]{MedSca} which is a 
				crucial ingredient in our paper:

				 \begin{theorem}\label{inv_curves}
				 	Let $f\in F[X]$ be a disintegrated polynomial
				 	of degree $d\geq 2$, let $n\geq2$ and let 
				 	$\displaystyle \varphi=(f,...,f):\ (\bP^1_F)^n\longrightarrow
				 	(\bP^1_F)^n$. Let $V$ be an irreducible $\varphi$-invariant 
				 	(respectively $\varphi$-periodic) 
				 	subvariety in 
				 	$(\bP^1_F)^n$. 
				 	For $1\leq i\leq n$, let
				 	$x_i$ be the chosen coordinate for the $i^{\text{th}}$ factor 
				 	of $(\bP^1)^n$. 
				 	Then
				 	$V$ is given by a collection of equations of the following types: 
				 		\begin{itemize}
				 			\item [(A)] $x_i=\zeta$ where $\zeta$
				 			is a fixed (respectively periodic) point of $f$.
				 			\item [(B)] $x_j=g(x_i)$ for some $i\neq j$, 
				 			where $g$ is a polynomial commuting with $f$ (respectively 
				 			an iterate of 
				 			$f$). 
				 		\end{itemize}
				 \end{theorem}

				We could further describe all the polynomials $g$ in type (B)
				of Theorem \ref{inv_curves} as follows. 
				\begin{proposition}\label{prop:all_g}
					Let $F$ and $f$ be as in Theorem \ref{inv_curves}. We have:
					\begin{itemize}
						\item [(a)] If $g\in F[X]$ has degree at least 2 such that 
						$g$
						commutes with an iterate of $f$ then $g$ and $f$
						have a common iterate.
						\item [(b)] Let $M(f^\infty)$
						denote the collection of all linear polynomials
						commuting with an iterate of $f$. Then $M(f^\infty)$
						is a finite cyclic group under composition.
						\item [(c)] Let $\tilde{f}\in F[X]$ be a polynomial
						of lowest degree at least 2 such that $\tilde{f}$
						commutes with an iterate of $f$. Then there exists
						$D=D_f>0$ relatively prime to
						the order of $M(f^\infty)$ such that 
						$\tilde{f}\circ L=L^D\circ \tilde{f}$
						for every $L\in M(f^\infty)$.
						\item [(d)] 
						$\left\{\tilde{f}^m\circ L:\ m\geq 0,\ L\in M(f^\infty)\right\}=\left\{L\circ\tilde{f}^m:\ m\geq 0,\ L\in M(f^\infty)\right\}$, and these sets describe exactly all
						polynomials $g$ commuting with an iterate of $f$.		
					\end{itemize}
				\end{proposition}
				\begin{proof}
					By the Lefschetz principle, we may assume $F=\C$. 
					Part (a) is a well-known result of Ritt \cite[p. 399]{Ritt}. 
					For part (b), let $\Sigma_f$ denote the group
					of linear fractional automorphism of the Julia set of $f$.
					It is known that $\Sigma_f$ is finite cyclic \cite{SchSte}.
					Therefore $M(f^\infty)$, being a subgroup of $\Sigma_f$,
					is also finite cyclic. By part (a), $f$ and $\tilde{f}$
					have the same Julia set. Therefore 
					$\Sigma_f=\Sigma_{\tilde{f}}$. By \cite{SchSte},
					there exists $D$ such that $\tilde{f}\circ L=L^{D}\circ \tilde{f}$ for every $L\in\Sigma_f=\Sigma_{\tilde{f}}$. To prove that
					$D$ is relatively prime to the order of $M(f^\infty)$,
					we let $\tilde{L}$ denote a generator of $M(f^\infty)$,
					and $N>0$ such that $\tilde{L}\circ \tilde{f}^N=\tilde{f}^N\circ L$. Hence $\tilde{L}\circ \tilde{f}^N=\tilde{L}^{D^N}\circ \tilde{f}^N$. The last equality
					implies $D^{N}-1$ is divisible by the order
					of $M(f^\infty)$ and we are done.

					It remains to show part (d). The given two sets are equal since
					$D^m$ is relatively prime to the order of $M(f^\infty)$
					for every $m\geq 0$. It suffices to show if $g\in F[X]$,
					$\deg(g)>1$ and $g$ commutes with $f$ then $g$ has
					the form $\tilde{f}^m\circ L$. Let $\varphi=(f,f)$
					be the split self-map of $(\bP^1_F)^2$. Now the
					(possibly reducible) curve $V$
					in $(\bP^1_F)^2$ given by $\tilde{f}(y)=g(x)$
					satisfies
					$\varphi^{M}(V)\subseteq V$ for some $M>0$. Therefore
					some irreducible component $C$ of $V$ is periodic.
					By Theorem \ref{inv_curves}, $C$ is given by 
					$y=\psi(x)$ or $x=\psi(y)$ where
					$\psi$ commutes with an iterate of $f$. Therefore one
					of the following holds:
					\begin{itemize}
						\item [(i)] $\tilde{f}\circ \psi=g$
						\item [(ii)] $g\circ \psi =\tilde{f}$
					\end{itemize}
					Since $\deg(g)\geq \deg(\tilde{f})$ by the definition of 
					$\tilde{f}$, case (ii) can only happen when 
					$\deg(g)=\deg(\tilde{f})$
					and $\psi\in M(f^\infty)$. If this is the case, we can write
					(ii) into $g=\tilde{f}\circ (\psi)^{-1}$. Thus
					we can assume (i) always happens. Repeating the argument 
					for the pair $(\tilde{f},\psi)$
					instead of $(\tilde{f},g)$,
					we get the desired conclusion.
				\end{proof}
				
				\begin{remark}
					Proposition \ref{prop:all_g} follows readily from Ritt's theory 
					of
					polynomial decomposition. The proof
					given here uses the Medvedev-Scanlon description in Theorem
					\ref{inv_curves} and simple results from complex dynamics.
					In fact, in an upcoming work, we will study and give
					examples of a lot of rational (and non-polynomial) maps
					$f$ such that Theorem \ref{inv_curves} is still valid. Then 
					an analogue of Proposition \ref{prop:all_g},
					especially part (d), still holds by exactly the same proof.   
				\end{remark}

				We conclude this section with a particularly useful property
				of preperiodic subvarieties of
				$(\bP^1_F)^n$. Let $f$, $n$ and $\varphi$ be as in Theorem \ref{inv_curves}. Let
				$V$ be an irreducible $\varphi$-periodic subvariety of $(\bP^1_F)^n$.
				We will associate to $V$ a binary relation $\prec$ on $I=\{1,\ldots,n\}$
				as follows. Let $I_V$ denote the set of $1\leq i\leq n$
				such that $V$ is contained in  a hypersurface of the form $x_i=\zeta$
				where $\zeta$ is a periodic point. The relation $\prec$ is empty if and only
				if $I_V=I$ (i.e. $V$ is a point). For every $i\in I-I_V$,
				we include the relation $i\prec i$.  
				For two elements $i\neq j$ in $I-I_V$, we include the relation
				$i\prec j$ if $V$ is contained in a hypersurface of the form
				$x_j=g(x_i)$ where $g$ is a polynomial commuting with an iterate of 
				$f$. We have the following properties: 
				\begin{lemma}\label{become_chain}
					Notations as in the last paragraph. Let $1\leq i,j,k\leq n$. We have:
					\begin{itemize}
						\item [(a)] Transitivity: if $i\prec j$ and $j\prec k$ then
						$i\prec k$.
						\item [(b)] Upper chain extension: if $i\prec j$ and $i\prec k$ then either
						$j\prec k$ or $k\prec j$.
						\item [(c)] Lower chain extension: if $i\prec k$ and $j\prec k$ then either
						$i\prec j$ or $j\prec i$.
					\end{itemize}
				\end{lemma}
				\begin{proof}
					We may assume $i$, $j$, and $k$ are distinct, otherwise there is
					nothing to prove.
					Part (a) is immediate from the definition of $\prec$. For part (b),
					we have that $V$ is contained in hypersurfaces $x_j=g_1(x_i)$
					and $x_k=g_2(x_i)$. By Proposition \ref{prop:all_g}, we may write
					$g_1=g_3\circ g_2$ or $g_2=g_3\circ g_1$ for some $g_3$
					commuting with an iterate of $f$. This implies $k\prec j$
					or $j\prec k$.
					
					Now we prove part (c). Let $\pi$ denote the projection
					from $(\bP^1)^n$ onto the $(i,j,k)$-factor $(\bP^1)^3$. We have that
					$\pi(V)$ is an irreducible $(f,f,f)$-periodic curve of $(\bP^1)^3$ 
					contained in the (not necessarily irreducible) curve
					given by $x_k=g_1(x_i)$ and $x_k=g_2(x_j)$ (note that we must have
					$\dim (\pi(V))>0$ since $i,j,k\notin I_V$). Now we consider the 
					closed embedding:
					$$(\bP^1_F)^2\stackrel{\eta}{\longrightarrow}(\bP^1_F)^3$$
					defined by $\eta(y_i,y_j)=(y_i,y_j,g_1(y_i))$. Now $\eta^{-1}(\pi(V))$
					is an irreducible $(f,f)$-periodic curve 
					of $(\bP^1_F)^2$ whose projection to each factor $\bP_1$ is surjective
					since $i,j\notin I_V$. Therefore $\eta^{-1}(\pi(V))$
					is given by either $y_i=g_3(y_j)$ or $y_j=g_3(y_i)$
					for some $g_3$ commuting with an iterate of $f$. This implies
					either $j\prec i$ or $i\prec j$.
				\end{proof}
				
				A chain is either a tuple of one element $(i)$ where 
				$i\notin I_V$ (equivalently $i\prec i$), or
				an ordered set of distinct elements $i_1\prec i_2\prec\ldots\prec i_l$.
				If $\cI=(i_1,\ldots,i_l)$ is a chain, we denote the underlying set
				(or the support)
				$\{i_1,\ldots,i_l\}$ by $s(\cI)$.
				Note that it is possible for many chains to have a common support, 
				for example if
				$V$ is contained in $x_j=g(x_i)$ where $g$ is linear then
				both $(i,j)$ and $(j,i)$ are chains. 
				By Lemma \ref{become_chain}, if $\cI$ is a chain, $i\in I$
				and $i\prec j$ or $j\prec i$ for some $j\in \cI$ then we can enlarge $\cI$
				into a chain whose support is $s(\cI)\cup\{i\}$. We have 
				that there exist maximal chains $\cI_1,\ldots,\cI_l$
				whose supports partition $I-I_V$. Although the collection
				$\{\cI_1,\ldots,\cI_l\}$ is not uniquely determined by $V$,
				the collection of supports $\{s(\cI_1),\ldots,s(\cI_l)\}$
				is. To prove these facts, one may define an equivalence relation
				$\approx$ on $I-I_V$ by $i\approx j$ if and only if 
				$i\prec j$ or $j\prec i$. Then it is easy to prove that
				$\{s(\cI_1),\ldots,s(\cI_l)\}$ is exactly the collection
				of equivalence classes. 
				
				For an ordered subset $J$ 
				of $I$, we define the following factor of $(\bP^1)^n$:
				$$(\bP^1)^J\colon =\prod_{j\in J} \bP^1.$$
				For a collection of ordered sets $J_1,\ldots,J_l$
				whose underlying (i.e. unordered) sets partition $I$,
				we have the canonical isomorphism:
				$$(\bP^1)^n=(\bP^1)^{J_1}\times\ldots\times (\bP^1)^{J_l}.$$ 
				We now have the following result:
				\begin{proposition}\label{product_of_curves}
					Let $f$ and $\varphi$ be as in Theorem \ref{inv_curves}. Let $V$
					be an irreducible $\varphi$-preperiodic subvariety of $(\bP^1_F)^n$.
					Assume that $\dim(V)>0$. Let $I=\{1,\ldots,n\}$, and let $I_V$ denote
					the set of all $i$'s such that $V$ is contained in a hypersurface
					of the form $x_i=\zeta_i$ where $\zeta_i$ is $f$-preperiodic.
					We fix a choice of an order on $I_V$, write $l=\dim(V)$.
					There exist a collection of ordered sets $J_1,\ldots,J_l$
					whose underlying sets partition $I-I_V$ such that under
					the canonical isomorphism
					$$(\bP^1)^n=(\bP^1)^{I_V}\times(\bP^1)^{J_1}\times\ldots\times (\bP^1)^{J_l},$$
					we have:
					$$V=(\prod_{i\in I_V} \{\zeta_i\})\times V_1\times\ldots\times V_l$$
					where $V_k$ is an $(f,\ldots,f)$-preperiodic curve of
					$(\bP^1)^{J_k}$ for $1\leq k\leq l$.   
				\end{proposition}
				\begin{proof}
					There exists $m$ such that $\varphi^m(V)$ is periodic. The conclusion of
					the Proposition for $\varphi^m(V)$ will imply the same conclusion for $V$,
					hence we may assume $V$ is periodic. We associate to $V$ a binary
					relation $\prec$ on $I$ as before. Then there exist maximal chains 
					$\cI_1,\ldots,\cI_l$ whose supports partition
					$I-I_V$. We now take $J_k=\cI_k$ for $1\leq k\leq l$. 
				\end{proof}

  \section{The Dynamical Hasse Principle} \label{Hasse}
    \subsection{Motivation and Main Results}\label{intro_Hasse}
    In this section, let $S$ be a fixed finite subset of $M_K$ containing all the 
    archimedean places. 
    For every variety $X$ over $K$, we define:
		\begin{equation}\label{define_XKS}
			X(K,S)=\prod_{v\notin S}X(K_v)
		\end{equation}
		equipped with the product topology, where each
		$X(K_v)$ is given the $v$-adic topology which is Hausdorff
		by separatedness of $X$.
		The set $X(K)$ is embedded into $X(K,S)$ diagonally. 
		For every subset $T$ of $X(K,S)$,
		write $\cC(T)$ to denote the closure of $T$ in $X(K,S)$. 
		The following theorem has been established by Poonen and Voloch \cite[Theorem A]{PoVo}:
		\begin{theorem}\label{PV}
			Assume that $K$ is a function field. Let $A$ be an abelian variety 
			and $V$ a subvariety of $A$ both defined over $K$. Then:
			\begin{equation}\label{PV_equation}
			V(K)=V(K,S)\cap \cC(A(K)).
			\end{equation}	
		\end{theorem}
		
		The analogue of Theorem \ref{PV} when $K$ is a number field is still 
		widely open. The main motivation for Poonen-Voloch theorem is the 
		determination
		of $V(K)$ especially when $V$ is a curve of genus at least 2 embedded
		into its Jacobian. More precisely, they are interested in the 
		Brauer-Manin obstruction to
		the Hasse principle studied by various authors. In fact, the idea of taking 
		the (coarser) intersection 
		between $V(K_\fp)$ and the $\fp$-adic closure of $A(K)$ in $A(K_\fp)$, 
		where
		$\fp$ is a prime of $K$, is dated back to Chabauty's work in the 1940s. 
		We refer the readers to \cite{PoVo} and the references there
		for more details.

		Now return to our general setting,
		let $\varphi$ be a $K$-morphism of $X$ to itself, $V$
			a closed subvariety 
			of $X$, and $P\in X(K)$ a $K$-rational point of $X$. 
		We have the following inclusion (note the similarity with (\ref{PV_equation})
		where the group $A(K)$ is
		replaced by the orbit $\cO_{\varphi}(P)$): 
				\begin{equation}\label{BM-subset}
				V(K)\cap \cO_{\varphi}(P)\subseteq V(K,S)\cap\cC(\cO_{\varphi}(P)).
				\end{equation}

		Motivated by the Poonen-Voloch theorem, Hsia and Silverman \cite[p. 237--238]{hsiasil_BM}
		ask:

		\begin{question}\label{BM1}
			Let $V^{\textpp}$ denote
			the union of all positive dimensional
			preperiodic subvarieties of $V$. 
			Assume that $\cO_{\varphi}(P)\cap V^{\textpp}(K)=\emptyset$. 
			When does equality hold in (\ref{BM-subset})?
		\end{question}
		
		The requirement $\cO_{\varphi}(P)\cap V^{\textpp}(K)=\emptyset$ is
		necessary as explained in \cite[p. 238]{hsiasil_BM}. In this paper,
		we restrict to the following question:
		

		 \begin{question}\label{BM2}
		 	Assume that $V$ is preperiodic and $V(K)\cap\cO_{\varphi}(P)=\emptyset$, when can we conclude
		 	 $V(K,S)\cap\cC(\cO_{\varphi}(P))=\emptyset$?
		 \end{question}  

			Our main theorems
			below will address Question \ref{BM2} when
			$X=(\bP^1)^n$, and $\varphi$ is the diagonally split morphism
			associated to a polynomial $f$. 
			We begin with
			the case $\dim(V)=0$:
			\begin{theorem}\label{intro_thm0}
				Let $f\in K[X]$ be a polynomial of degree at least 2, let $n\geq2$
				be an integer, 
				and let $\varphi$ denote the split morphism 
				$(f,\ldots,f)\colon (\bP^1_K)^n\longrightarrow (\bP^1_K)^n$. 
				Let $V$
				be a zero dimensional subvariety of $(\bP^1_K)^n$. The following hold:
				\begin{itemize}
				  \item [(a)]  
				  For every $P\in (\bP^1)^n(K)$ such that
				  	$V(K)\cap\cO_{\varphi}(P)=\emptyset$, 
				  	there exist infinitely many primes $\fp$
				  	such that $V(K_\fp)$ does not intersect the $\fp$-adic closure 
				  	of $\cO_{\varphi}(P)$. 
				  
				  \item [(b)] Question \ref{BM1} has an affirmative answer, namely
				  for every $P\in X(K)$ we have:
				  	$$V(K)\cap\cO_{\varphi}(P)=V(K,S)\cap\cC(\cO_{\varphi}(P)).$$
				  
				  \item [(c)] In this part only, we assume $f$ is \textbf{special} and $V$
				  is \textbf{preperiodic}. Then for every
				  $P\in (\bP^1)^n(K)$ such that $V(K)\cap\cO_{\varphi}(P)=\emptyset$,
				  for \textbf{almost all} primes $\fp$ of $K$, we have $V(K_\fp)$
				  does not intersect the $\fp$-adic closure
				  of $\cO_{\varphi}(P)$.

				\end{itemize}
			\end{theorem}

			Part (b) actually holds
			for maps of the form $(f_1,\ldots,f_n)$
			where each $f_i$ is an arbitrary rational map
			of degree at least 2. This
			more general result follows from the main
			results of Silverman and Voloch \cite{SilVo}.
			We will see that the trick used to establish
			part (a) in Subsection \ref{Prelim}, which is similar to
			one used in \cite{SilVo}, appears repeatedly in this section and
			can be modified to reduce our problem (when $\dim(V)>0$) to the \'etale case
			(see Subsection \ref{disint_case}). 
			Part (c) of Theorem \ref{intro_thm0}
			could be generalized completely, we have:

			\begin{theorem}\label{intro_thm3}
			Let $f\in K[X]$ be a \textbf{special}
			polynomial of degree $d\geq 2$. Let  $n\geq 2$, and 
			$\varphi=(f,\ldots,f)$ be as in Theorem \ref{intro_thm0}.
			Let $V$ be a subvariety of $(\bP^1_K)^n$ such
			that every irreducible component of $V_{\bar{K}}$
			is a preperiodic subvariety. 
			Let $P\in (\bP^1)^n(K)$ such that 
			$V(K)\cap \cO_{\varphi}(P)=\emptyset$.
			Then for \textbf{almost all} primes $\fp$ of $K$,
			$V(K_\fp)$ does not intersect the
			$\fp$-adic closure of $\cO_{\varphi}(P)$.
			Consequently, Question \ref{BM2} has an affirmative answer:
			$V(K,S)\cap\cC(\cO_{\varphi}(P))=\emptyset$.
			\end{theorem}
			
						 It has been known since the beginning of the theory
			of complex dynamics that special polynomials and disintegrated
			polynomials have
			very different dynamical behaviours. When
			$f$ is disintegrated, we are still able
			to prove that a Hasse principle analogous to
			Theorem \ref{intro_thm3} holds 
			when $V$ is a curve or a hypersurface: 
						
			\begin{theorem}\label{intro_thm1}
			Let $f\in K[X]$ be a \textbf{disintegrated} polynomial of degree $d\geq 2$. 
			Let $n\geq 2$, and $\varphi=(f,\ldots,f)$ be as in Theorem \ref{intro_thm0}.
			Let $V$ be a $\varphi$-preperiodic
			and absolutely irreducible \textbf{curve} or \textbf{hypersurface} of $(\bP^1_K)^n$. Let $P\in (\bP^1)^n(K)$
			 such that 
			$V(K)\cap\cO_{\varphi}(P)=\emptyset$.
			Then for infinitely many primes $\fp$ of $K$,
			the $\fp$-adic closure of $\cO_{\fp}(P)$
			does not intersect $V(K_\fp)$.
			Consequently, Question \ref{BM2} has an affirmative
			answer: we have $V(K,S)\cap\cC(\cO_{\varphi}(P))=\emptyset$.					
			\end{theorem}

			Although we expect Theorem \ref{intro_thm1} still holds for an arbitrary
			absolutely irreducible preperiodic subvariety $V$ (i.e. $1<\dim(V)<n-1$), 
			we need to assume
			an extra technical assumption, as follows:
			\begin{theorem}\label{intro_thm1*}
				Let $f$, $n$, and $\varphi$ be as in Theorem \ref{intro_thm1}.
				Assume the technical assumption 
				that every polynomial commuting with an iterate
				of $f$ also commutes with $f$ 
				.
				Let $V$ be an absolutely irreducible $\varphi$-preperiodic
				subvariety of $(\bP^1_K)^n$. Let $P\in (\bP^1_K)^n$
				such that $V(K)\cap \cO_{\varphi}(P)=\emptyset$. Then there exist
				infinitely many primes $\fp$ of $K$ such that the $\fp$-adic closure
				of $\cO_{\varphi}(P)$ does not intersect $V(K_\fp)$. Consequently,
				Question \ref{BM2} has an affirmative answer: $V(K,S)\cap 
				\cC(\cO_{\varphi}(P))=\emptyset$.
			\end{theorem}
			
			\begin{remark}
			The above technical assumption holds for a generic $f$. 
			In fact, let $M(f^\infty)$ denote the group of linear
			polynomials commuting with an iterate of $f$. By Proposition 
			\ref{prop:all_g}, if $M(f^\infty)$ is trivial
			then the technical assumption in Theorem \ref{intro_thm1*}
			holds. 
			When $f$ has degree 2
			and is not conjugate to $X^2$, we have that $M(f^\infty)$ is trivial.
			When $f$ has degree at least 3, after making a linear
			change, we can assume:
			$$f(x)=X^d+a_{d-2}X^{d-2}+a_{d-3}X^{d-3}+\ldots+a_0.$$
			It is easy to prove
			that when $a_{d-2}a_{d-3}\neq 0$, 
			the group $M(f^\infty)$ is trivial.
			\end{remark}

      In the next subsection, we will give all the preliminary 
			results
			needed for the proofs of the above Theorems
			as well as a proof of Theorem \ref{intro_thm0}. 
						
			\subsection{An Assortment of Preliminary Results}\label{Prelim}
			Our first lemma shows that in order to prove 
			Theorems \ref{intro_thm0}--\ref{intro_thm1*}, we are free to
			replace $K$ by a finite extension.
			
			\begin{lemma}\label{ext}
			Let $L$ be a finite extension of $K$, $X$ a variety over $K$,
			$\varphi$ a $K$-endomorphism of $X$, $V$ a closed subvariety of 
			$X$ over $K$, and $P$ an element of $X(K)$. Let $\fp$ be a prime
			of $K$ and $\fq$ a prime
			of $L$ lying above $\fp$. If $V(L_\fq)$ does not intersect the
			$\fq$-adic closure of $\cO_\varphi(P)$ in $X(L_\fq)$
			then $V(K_\fp)$ does not intersect the $\fp$-adic closure
			of $\cO_\varphi(P)$ in $X(K_\fp)$.
			\end{lemma}
			\begin{proof}
			Clear.
			\end{proof}


			Before stating the next result, we need some terminology. Let
			$\fp$ be a prime of $K$, $\scrX$ a
			separated scheme of finite type over $\scrO_\fp$. By 
			the valuative criterion of separatedness \cite[p. 97]{Hartshorne},
			we could view $\scrX(\scrO_\fp)$ as a subset of
			of $\scrX(K_\fp)$, then the 
			$\fp$-adic topology on $\scrX(\scrO_\fp)$
			is the same as the subspace topology induced
			by the $\fp$-adic topology on $\scrX(K_\fp)$. Every point
			$P\in\scrX(\scrO_\fp)$ is an $\scrO_\fp$-morphism
			$\Spec(\scrO_\fp)\longrightarrow \scrX$. By the generic point
			and closed point of $P$, we mean the image of
			the generic point and closed point 
			of $\Spec(\scrO_\fp)$, respectively. We write
			$\bar{P}$ to denote its closed point, which is also
			identified to the corresponding element in $\scrX(k_\fp)$.
			The scheme $\scrX$ is said to be smooth at $P$
			if the structural morphism $\scrX\longrightarrow\Spec(\scrO_\fp)$
			is smooth at $\bar{P}$. 
			Similarly, an
			endomorphism $\varphi$ of $\scrX$ over $\scrO_\fp$
			is said to be \'etale at $P$ if it is
			\'etale at $\bar{P}$. 
			The following is essentially a main result of \cite[Theorem 4.4]{AKNTVV}:	
				\begin{theorem}\label{etale}
					Let $K$, $\fp$, $\scrX$, and $P\in\scrX(\scrO_\fp)$
					 be as in the last paragraph. 
					Let $\varphi$ be an endomorphism of $\scrX$ over $\scrO_\fp$.
					Assume that
					$\scrX$ is smooth and $\varphi$ is \'etale
					at every point in the orbit $\cO_{\varphi}(P)$. Let
					$\scrV$ be a reduced closed subscheme of
					$\scrX$. Assume one of the following sets of conditions:
					\begin{itemize}
					\item [(a)] There exists
					$M>0$ satisfying 
					$\varphi^{M}(\scrV)\subseteq\scrV$.
					When $K$ is a function field, we
					assume that $P$ is $\varphi$-preperiodic modulo $\fp$ 
					. 
					\item [(b)] $\scrV$ is a finite
					set of preperiodic points of $\scrX(\scrO_\fp)$. 				
				\end{itemize}
					
				We have: if $\scrV(\scrO_\fp)$
					does not intersect $\cO_{\varphi}(P)$
					then it does not intersect the $\fp$-adic closure of
					$\cO_{\varphi}(P)$.
				\end{theorem}
				\begin{proof}
					First assume the conditions in (a). 
					Although the statement in \cite[Theorem 4.4]{AKNTVV}
					includes smoothness of $\scrX$ and \'etaleness of $\varphi$
					everywhere, its proof could  actually be carried verbatim here. 
					
					Now assume the conditions in (b). Define:
					$$\scrV_1=\bigcup_{i=0}^{\infty} \varphi^i(\scrV).$$
					Then $\scrV_1$ is a finite set of points in $\scrX(\scrO_\fp)$
					satisfying $\varphi(\scrV_1)\subseteq \scrV_1$.
					If the orbit of $P$ intersects $\scrV_1(\scrO_\fp)$
					then $P$ is preperiodic and there is nothing to prove.
					So we may assume otherwise. After reducing mod $\fp$, if 
					the orbit of
					$P$ does not intersect $\scrV_1$ then there is nothing
					to prove. So we may assume otherwise, and this assumption
					gives that $P$ is preperiodic mod $\fp$. All
					the conditions in part (a) are now satisfied, and we can
					get the desired conclusion.
				\end{proof}
				
				We remind the readers that 
			if $K$ is a function field over the constant field $\kappa$, a
			rational function $f\in K(X)$ is said to be isotrivial
			if there exists a fractional linear map $L\in\Aut(\bP^1(\bar{K}))$
			such that $L^{-1}\circ f\circ L\in \bar{\kappa}(X)$. The Silverman-Voloch trick mentioned right 
			after
				Theorem \ref{intro_thm0} is the following (see \cite{IS} for all the terminology):
				\begin{lemma}\label{good_p}
					Let $f\in K[X]$ be a polynomial of degree at least 2, and
					let $\alpha\in K$ such that the canonical height
					$\hat{h}_f(\alpha)$ is positive.
					Let $\gamma\in K$ be a periodic point of $f$ such that
					$f$ is not of polynomial type at $\gamma$ 
					.
					Then there are
					infinitely many primes $\fp$ of $K$ such that
					$v_{\fp}(f^{\mu}(\alpha)-\gamma)>0$ for some $\mu$
					depending on $\fp$.
				\end{lemma}
				\begin{proof}
					This follows from the deeper result of Ingram-Silverman
					\cite[p. 292]{IS} that almost all elements of the sequence
					$(f^{\mu}(\alpha)-\gamma)$ have a primitive
					divisor. In the function field case, Ingram and Silverman
					require that $f$ is not isotrivial \cite[Remark 4]{IS}. However,
					even when $f$ is isotrivial, what is really needed in the proof of their result
					is that $\hat{h}_f(\alpha)>0$.					
				\end{proof}
				
				\begin{remark}
					If the exact $f$-period of $\gamma$ is greater than 2 then $f$ is not 
					of polynomial type at $\gamma$ \cite[Remark 6]{IS}.
				\end{remark}
				
				We can use Lemma \ref{good_p} to prove the following: 
				\begin{lemma}\label{Trick}
					Let $f$ be as in Lemma \ref{good_p}. Let $\alpha$ be an element
					of $\bP^1(K)$ and $V$ a finite subset of $\bP^1(K)$
					such that the orbit of $\alpha$ does not intersect $V$.
					Then there are infinitely many primes $\fp$ such that
					the $\fp$-adic closure of the orbit of $\alpha$ 
					does not intersect $V$. 
				\end{lemma}
				\begin{proof}
					If $\alpha$ is preperiodic, there is nothing to prove. 
					We
					assume that $\alpha$ is wandering. If 
					$\hat{h}_f(\alpha)=0$, we must have that $K$ is a function field and $f$ is
					isotrivial \cite{Bene}. After replacing $K$ by a finite extension, and 
					making a linear change, we may assume that $f\in \kappa[X]$ and $\alpha\in \kappa$.
					Now the conclusion of the lemma is obvious since the orbit of
					$\alpha$ is discrete in the $\fp$-adic topology for every $\fp$. 
				 
				 What remains now is the case $\hat{h}_f(\alpha)>0$.
					 For almost
					all $\fp$, we have $v_{\fp}(\alpha)\geq 0$ and $f\in\scrO_{\fp}[X]$,
					so $v_\fp(f^m(\alpha))\geq 0$ for all $m$. Therefore we can assume
					$\infty\notin V$. Let $u_1,...,u_q$ be all elements of $V$. By Lemma
					\ref{ext}, we can assume there is a periodic point
					$\gamma\in K$ of exact period at least 3 and the orbit 
					of $\gamma$ does not contain $u_i$ for $1\leq i\leq q$.
					Hence there is a finite set of primes $T$ such that: 
					\begin{equation}\label{trick_1}
				 		\text{$f\in\scrO_{\fp}[X]$ and
				  $v_{\fp}(u_i-f^{m}(\gamma))=0\ \forall m\geq 0\ \forall 1\leq i\leq q$ 
				  $\forall\fp\notin T$.}
				 	\end{equation}
					Now by Lemma \ref{good_p}, there are infinitely many primes 
				 	$\fp\notin T$ such that:
				 	\begin{equation}\label{trick_2}
				 		\text{$v_\fp(f^{\mu}(\alpha)-\gamma)>0$ for some $\mu=\mu_\fp$}
				 	\end{equation}  
				  Fix any $\fp\notin T$ that gives (\ref{trick_2}), write $\mu=\mu_\fp$.
				  Thus $v_\fp(f^m(\alpha)-f^{m-\mu}(\gamma))>0$ for every $m\geq \mu$.
				  Together with (\ref{trick_1}), we have $v_\fp(f^m(\alpha)-u_i)=0$
				  for every $m\geq \mu$ and $1\leq i\leq q$. This implies that
				  $V$ does not intersect the $\fp$-adic closure
				  of the $f$-orbit of $\alpha$.				
				\end{proof}
				
				Now
				we have all the results needed to prove
				Theorem \ref{intro_thm0}:
				 	
				 	\noindent\textit{Proof of Theorem \ref{intro_thm0}}:
				 	By Lemma \ref{ext}, we can replace $K$ by a finite extension
				 	so that $V$ is a finite set of points in $(\bP^1)^n(K)$.
				 	
				 	For part (a), note that if $P$ is $\varphi$-preperiodic then there is 
				 	nothing to prove, hence we can assume $P$ is wandering. Write 
				 	$P=(\alpha_1,\ldots,\alpha_n)$, without loss of generality, we assume $\alpha_1$
				 	is wandering with respect to $f$. Let $U$
				 	denote the finite subset of $\bP^1(K)$ consisting
				 	of the first coordinates of points in $V$. There is the largest
				 	$N$ such that $f^N(\alpha_1)\in U$. We simply
				 	replace $P$ by $\varphi^{N+1}(P)$
				 	and assume the $f$-orbit of $\alpha_1$ does
				 	not contain any element of $U$. Then our conclusion
				 	follows from Lemma \ref{Trick}.

				 	Part (b) follows easily from part (a). As before, we can 
				  assume $P$ is
				  not preperiodic, hence 
				  there is the largest $N$ such that $\varphi^N(P)\in V(K)$.
				  Replacing $P$ by $\varphi^{N+1}(P)$, we can assume that 
				  $V(K)\cap \cO_\varphi(P)=\emptyset$, then part (a)
				  implies 
				  $$V(K,S)\cap\cC(\cO_{\varphi}(P))=\emptyset=V(K)\cap \cO_{\varphi}(P).$$

				 	For part (c), we first consider the case $L\circ f\circ L^{-1}=X^d$
				 	for some linear polynomial $L\in \bar{K}[X]$. 
				 	By extending $K$, we may assume  
				 	$L\in K[X]$. Since $L$ yields a homeomorphism 
				 	from $\bP^1(K_\fp)$ to itself for almost all $\fp$, we may
				 	assume $f(X)=X^d$. As before, we can assume the first coordinate 
				 	$\alpha_1$ of $P$ is wandering. Since $U$
				 	contains only $f$-preperiodic points (by preperiodicity of $V$), the $f$-orbit of $\alpha_1$ 
				 	does not contain any element of $U$. 
					For almost all $\fp$,
				 	the first coordinates of points in the $\varphi$-orbit of $P$ is
				 	a $\fp$-adic unit. Therefore we can exclude from $V$ all the points
				 	having first coordinates $0$ or $\infty$, hence $U\subseteq \Gm(K)$. 
				 	Let $\fp$ be a prime not dividing $d$ such that $\alpha_1$ 
				 	and all elements of $U$ are $\fp$-adic units. We now apply
				 	Theorem 2.10 for $\scrX=\Gm$ over $\scrO_\fp$, 
				 	$\scrV=U\subseteq\Gm(\scrO_\fp)$, the self-map being the 
				 	multiplication-by-$d$ 
				 	map, and the orbit of $\alpha_1$. Since the $\fp$-adic closure of
				 	the orbit of $\alpha_1$ does not intersect $U(K_\fp)$,
				 	the $\fp$-adic closure of $P$ does not intersect $V(K_\fp)$.  
				  
				  For the case $L\circ f\circ L^{-1}=\pm C_d(X)$, we use the 
				  self-map of $(\bP^1)^n$ given by:
				  $$(x_1,\ldots,x_n)\mapsto(x_1+\frac{1}{x_1},\ldots,x_n+\frac{1}{x_n})$$
				  to reduce to the case that $f$ is conjugate to $\pm X^d$ 
				  which has just been treated. This finishes the proof of Theorem
				  \ref{intro_thm0}.


				 \subsection{Proof of Theorem \ref{intro_thm1}}\label{disint_case} 		
				 Let $f\in K[X]$ be a disintegrated polynomial
				 of degree
				 $d\geq 2$. 
				 By Theorem \ref{inv_curves}, for every preperiodic hypersurface
				 $H$ of $(\bP^1_K)^n$, there exist $1\leq i<j\leq n$
				 such that $H=\pi^{-1}(C)$
				 where $\pi$ denotes the projection onto the $(i,j)$-factor
				 and $C$ is an $(f,f)$-preperiodic curve of $(\bP^1_K)^2$.
				 Therefore it suffices to prove Theorem \ref{intro_thm1}
				 when $V$ is a curve. 
				 
				 Let $\varphi$, $V$ and $P\in (\bP^1)^n(K)$ such that
				 $V(K)\cap \cO_{\varphi}(P)=\emptyset$ and
				 $\dim(V)=1$ as
				 in Theorem \ref{intro_thm1} and the discussion in the last paragraph.
					Let $I$ and $I_V$ be as in Proposition \ref{product_of_curves}.
	       Let $\tilde{f}$,
				 $M(f^\infty)$, and $D=D_f$ be as in Proposition 
				 \ref{prop:all_g}.
				 Now we prove that there are infinitely many primes $\fp$ 
				 of $K$ such that $V(K_\fp)$ does not intersect the 
				 $\fp$-adic closure of $\cO_\varphi(P)$. By Lemma \ref{ext},
				 we can assume $\tilde{f}$ and all elements
				 in $M(f^\infty)$ have coefficients in $K$.  
				 
				 \textbf{Step 1:} We first consider the
				 case $V$ is periodic.

		%
				  
				  \textbf{Step 1.1:} we assume that $I_V=\emptyset$. By Theorem
				  \ref{inv_curves}, Proposition
				  \ref{product_of_curves} and the discussion before it, we can relabel
				  the factors of $(\bP^1)^n$ and rename the coordinate
				  functions of all the factors as $x$, $y_1,\ldots,y_{n-1}$  
				  such that $V$ is given by the equations: $y_i=g_i(x)$
				  for $1\leq i\leq n-1$, where $g_i$ commutes with an iterate of
				  $f$ for $1\leq i\leq n-1$ and $\deg(g_1)\leq\ldots\leq\deg(g_{n-1})$.
				  Write $P=(a,b_1,\ldots,b_{n-1})$.


				  \textbf{Step 1.1.1:} we consider the easy case that $a$
				  is $f$-preperiodic. 
				  Replacing $P$ by an iterate, we can assume that 
				  $a$ is $f$-periodic of exact period $N$. 
				  The $\varphi$-orbit of $P$ is:
				  $$\{(f^i(a),f^{i+tN}(b_1),\ldots,f^{i+tN}(b_{n-1})):\ t\geq 0,\ 
				  0\leq i<N\}.$$
				  Since this orbit does not intersect $V(K)$, we have
				  \begin{equation}\label{eq:logic}
				  \forall t\geq 0\ \forall 0\leq i<N\ \exists 1\leq j\leq n-1\ 
				  \left(f^{i+tN}(b_j)\neq g_j(f^i(a))\right).
				  \end{equation}
				   
				  For each $0\leq i<N$ and $1\leq j\leq n-1$, denote
				  $\cB_{i,j}=(f^i)^{-1}(\{g_j(f^i(a))\})$. 
				  Denote 
				  $$\cB=\bigcup_{0\leq i<N}\cB_{i,1}\times\ldots\times\cB_{i,n-1}$$
				  which is a finite set of (preperiodic) points of $(\bP^1)^{n-1}$.
				  Let $b=(b_1,\ldots,b_{n-1})$
				  and let $\phi$ denote the self-map $(f,\ldots,f)$
				  of $(\bP^1)^{n-1}$. By (\ref{eq:logic}), we have 
				  $\phi^{tN}(b)\notin \cB$ for every $t\geq 0$. 
				  By Theorem \ref{intro_thm0}, there exist
				  infinitely many primes $\fp$ such that the $\fp$-adic closure
				  $\cC_{\fp}$ of $\{\phi^{tN}(b):\ t\geq 0\}$
				  does not intersect $\cB$.
				  For each such $\fp$, the $\fp$-adic closure of
				  the orbit of $P$ lies in:
				  $$\bigcup_{0\leq i<N} \{f^i(a)\}\times \phi^i(C_\fp)$$
				  which is disjoint from $V(K_\fp)$.

					\textbf{Step 1.1.2:} we consider the case $a$ is $f$-wandering
					and $\hat{h}_f(a)=0$. Hence $K$ is
					a function field
					and $f$ is isotrivial by \cite{Bene}. After replacing $K$ by a finite
					extension, and making a linear change, we may assume that
					$f\in \kappa[X]$ and $a\in\kappa$. If $\hat{h}_{f}(b_i)=0$
					for every $1\leq i\leq n-1$, then $b_i\in\kappa$ for every
					$1\leq i\leq n-1$. Then the orbit $\cO_{\varphi}(P)$
					is discrete in the $\fp$-adic topology for every $\fp$, and the 
					conclusion of the theorem is obvious. Hence we assume that
					there is $1\leq j\leq n-1$ such that $\hat{h}_f(b_j)>0$.
					Replacing $K$ by a finite extension if necessary, we assume that there
					is an $f$-periodic $\gamma\in \kappa$ 
					of exact period at least 3. By Lemma \ref{good_p}, there is an infinite
					set of primes $T$ such that for every $\fp\in T$:
					\begin{equation}\label{eq:iso1}
					v_{\fp}(f^{\mu}(b_j)-\gamma)>0\ \text{for some $\mu=\mu_\fp$.}
					\end{equation}
					
					For each $\fp\in T$, if the $\fp$-adic closure of $\cO_{\varphi}(P)$
					intersects $V(K_\fp)$ then 
					we must have:
					\begin{equation}\label{eq:iso2}
					v_{\fp}(f^l(b_j)- g_j(f^l(a)))>0\ \text{for some $l=l_\fp\geq \mu_\fp$}.
					\end{equation}
					
					From (\ref{eq:iso1}) and (\ref{eq:iso2}), we have:
					$$v_{\fp}(f^{l-\mu}(\gamma)- g_j(f^l(a)))>0.$$
					Since $a,\gamma\in \kappa$ and $f\in\kappa[X]$, this
					equality means $f^{l-\mu}(\gamma)= g_j(f^l(a))$. Hence
					$a$ is $f$-preperiodic, contradiction. Therefore, for every $\fp\in T$,
					the $\fp$-adic closure of $\cO_{\varphi}(P)$ does not intersect $V(K_\fp)$.
					
				  \textbf{Step 1.1.3:} we turn to the most difficult case, namely $\hat{h}_f(a)>0$.
				  By Proposition \ref{prop:all_g},
				  write $f=\rho_1\circ\tilde{f}^A$ for some 
				  integer $A\geq 1$ and linear $\rho_1\in M(f^\infty)$.
				  Since $D$ is relatively prime to $|M(f^\infty)|$,
				  there is $\rho_2\in M(f^\infty)$
				  such that $f=(\tilde{f}\circ \rho_2)^A$. Replacing
				  $\tilde{f}$ by $\tilde{f}\circ \rho_2$,
				  we may assume $f=\tilde{f}^A$.
				  For each $1\leq j\leq n-1$, write
				  $g_j=L_j\circ \tilde{f}^{m_j}$.
				  
			    For almost all
				  $\fp$, we have $v_{\fp}(f^n(a))\geq 0$
				  for every $n\geq 0$. If for some $1\leq j\leq n-1$,
				  $b_j=\infty$ then for almost all $\fp$, the 
				  $\fp$-adic closure of the orbit
				  of $P$ lies in:
				  $$\left\{(x,y_1,\ldots,y_{j-1},\infty,y_{j+1},\ldots,y_{n-1}):\ x\in K_\fp,\ v_\fp(x)\geq 0\right\}$$
				  which is disjoint
				  from $V(K_\fp)$. So we can assume $b_j\neq\infty$ for every $1\leq j \leq n-1$. 
				  
				  By taking a finite extension of
				  $K$ if necessary, we choose an $\tilde{f}$-periodic 
				  point $\gamma\in K$
				  of exact period $N\geq3$ such
				  that every point in the $\tilde{f}$-orbit of $\gamma$ is not
				  a zero of the derivative $\tilde{f}'(X)$ of $\tilde{f}(X)$.
				  By Lemma \ref{good_p}, there is an infinite set of primes
				  $R$ such that for every $\fp\in R$, all of the following hold:
				  \begin{equation}\label{con3}
				   \text{$a,b_1,\ldots,b_{n-1}\in\scrO_{\fp}$, 
				   in other words $P\in\bA^n(\scrO_{\fp})$}
				  \end{equation}
				  \begin{equation}\label{con4}
				  	v_{\fp}(\tilde{f}'(\tilde{f}^i(\gamma)))=0\ \forall\ 0\leq i<N
				  \end{equation}
				  \begin{equation}\label{con5}
				  	\text{$\tilde{f},f,L_1,\ldots,L_{n-1}$ are in $\scrO_{\fp}[X]$ and their leading
				  	coefficients are $\fp$-adic units}
				  \end{equation}
				  \begin{equation}\label{con6}
				  	\text{$v_{\fp}(\tilde{f}^{\mu}(a)-\gamma)>0$ for some 
				  		$\mu=\mu_{\fp}$.}
				  \end{equation}
				  
				  In fact, for $1\leq j\leq n-1$, the leading coefficient $c_j$
				  of $L_j$ is a root of unity
				  and $f=\tilde{f}^A$ is 
				  an iterate of $\tilde{f}$, hence the conditions on $c_j$
				  and $f$
				  in (\ref{con5}) are redundant. 
				  Now fix a prime $\fp$ in $R$ and write $\mu=\mu_\fp$, we still use
				  $V$ to denote the model $y_j=L_j\circ \tilde{f}^{m_j}(x)\forall j$
				  over $\scrO_\fp$, hence it makes sense to write $V(\scrO_\fp)$
				  and $V(k_\fp)$. We also use $P$, and $\varphi$
				  to denote the corresponding models over $\scrO_\fp$.
				  Replacing $P$ by $\varphi^{\mu}(P)$ and 
				  $\gamma$ by $\tilde{f}^{A\mu-\mu}(\gamma)$ if necessary, we can
				  assume that $v_{\fp}(a-\gamma)>0$. This gives that
				  $a$ is $\tilde{f}$-periodic, hence
				  $f$-periodic, modulo $\fp$ and:
				  \begin{equation}\label{con7}
				  \text{$v_{\fp}(\tilde{f}^l(a)-\tilde{f}^l(\gamma))>0$ and 
				  	$v_{\fp}(\tilde{f}'(\tilde{f}^l(a))-\tilde{f}'(\tilde{f}^l(\gamma)))>0\ \forall l\geq0$}
				  \end{equation}
				  The second inequality in (\ref{con7}) together with (\ref{con4})
				   give:
				  \begin{equation}\label{con8}
				  	v_{\fp}(\tilde{f}'(\tilde{f}^l(a)))=0\ \forall l\geq0
				  \end{equation}
				  
				  By (\ref{con8}) and induction, we have:
				  \begin{equation}\label{con9}
				  v_{\fp}((\tilde{f}^l)'(\tilde{f}^k(a)))=0\ \forall 
				  l\geq 0\ \forall k\geq 0
				  \end{equation}
				  
				  Since $f=\tilde{f}^A$, identity (\ref{con9}) implies:
				  \begin{equation}\label{con9bis}
				  v_{\fp}((f^l)'(f^k(a)))=0\ \forall 
				  l\geq 0\ \forall k\geq 0
				  \end{equation}

				  Since the $\varphi$-orbit of $P$ lies in $\bA^n(\scrO_\fp)$
				  which is closed in $(\bP^1)^n(K_\fp)$,
				  it suffices to show that $V(\scrO_\fp)$ does
				  not intersect the $\fp$-adic closure of the $\varphi$-orbit of
				  $P$. 
				  Assume there is $\eta$ such that the mod $\fp$ reduction				 				  $\varphi^\eta(\bar{P})$ lies in $V(k_{\fp})$, otherwise
				  there is nothing to prove. After replacing $P$ by $\varphi^\eta(P)$,
				  we can assume $\eta=0$, or in other words $\bar{P}\in V(k_\fp)$.
				  This means 
				  \begin{equation}\label{con10}
				  v_{\fp}(b_j-L_j\circ \tilde{f}^{m_j}(a))>0\ \forall 1\leq j\leq n-1. 
				  \end{equation}
				  Note that
				  $\tilde{f}\circ L=L^D\circ \tilde{f}$, and $L$ has finite order, therefore (\ref{con10}) together
				  with the $\tilde{f}$-periodicity mod $\fp$ of $a$
				  give that $b_j$ is
				  $\tilde{f}$-preperiodic, hence
				  $f$-preperiodic, mod $\fp$ for $1\leq j\leq n-1$. Therefore 
				  $P$ is $\varphi$-preperiodic
				  mod $\fp$.

				  Inequality (\ref{con10}) shows that:
				  \begin{equation}\label{con11}
				  	v_{\fp}(\tilde{f}^l(b_j)-L_j^{D^l}\circ \tilde{f}^{l+m_j}(a))>0\ \forall l\geq 0\ \forall 1\leq j\leq n-1.
				  \end{equation}
				  Our next step is to show:
				  \begin{equation}\label{con12}
				  	v_{\fp}(\tilde{f}'(L_j^{D^l}\circ \tilde{f}^{l+m_j}(a)))=0\ \forall l\geq0\ \forall 1\leq j\leq n-1
				  \end{equation}
				  Recall that $c_j$ denotes the leading coefficient of 
				  the 
				  linear polynomial $L_j$, we have:  
				  \begin{equation}\label{con13}
				  (\tilde{f}\circ L_j^{D^l}\circ \tilde{f}^{m_j+l})'(a)=(L_j^{D^{l+1}}\circ \tilde{f}^{m_j+l+1})'(a)=c_j^{D^{l+1}}(\tilde{f}^{m_j+l+1})'(a)
				  \end{equation}
				  and
				  \begin{equation}\label{con14}
				  (\tilde{f}\circ L_j^{D^l}\circ \tilde{f}^{m_j+l})'(a)=\tilde{f}'(L_j^{D^l}\circ \tilde{f}^{m_j+l}(a))
				  	c_j^{D^l}(\tilde{f}^{m_j+l})'(\alpha)				  	
				  \end{equation}
				  Since $c_j$ is a $\fp$-adic unit,
				  (\ref{con13}) and (\ref{con14}) imply:
				  \begin{equation}\label{con14bis}
				  v_{\fp}((\tilde{f}^{m_j+l+1})'(a))
				  =v_{\fp}(\tilde{f}'(L_j^{D^l}\circ \tilde{f}^{m_j+l}(a))
				  	(\tilde{f}^{m_j+l})'(\alpha))				  	
				  \end{equation}
				  Now (\ref{con12}) follows from (\ref{con9}), and
				  (\ref{con14bis}).  
				  By (\ref{con11}) and (\ref{con12}), we have:
				  \begin{equation}\label{con15}
				  	v_{\fp}(\tilde{f}'(\tilde{f}^l(b_j)))=0\ \forall l\geq 0\ \forall 1\leq j\leq n-1
				  \end{equation}
				  By (\ref{con15}) and induction, we have:
				  \begin{equation}\label{con15'}
				  v_{\fp}((\tilde{f}^l)'(\tilde{f}^k(b_j)))=0\ \forall 
				  l\geq 0\ \forall k\geq 0\ \forall 1\leq j\leq n-1
				  \end{equation}
				  Since $f=\tilde{f}^A$, identity (\ref{con15'}) implies:
				  \begin{equation}\label{con15''}
				  v_{\fp}((f^l)'(f^k(b_j)))=0\ \forall 
				  l\geq 0\ \forall k\geq 0\ \forall 1\leq j\leq n-1
				  \end{equation}
				  
				  Now (\ref{con9bis}) and (\ref{con15''}) show that 
				  the $\scrO_{\fp}$-morphism $\varphi$ is \'etale at every 
				  $\scrO_\fp$-valued point in the orbit of $P$. Together
				  with the fact that
				  $P$ is preperiodic mod $\fp$, we can apply
				   Theorem \ref{etale} to get the desired conclusion. This
				   finishes the case $V$ is periodic and $I_V=\emptyset$.

				   \textbf{Step 1.2:} assume that $I_V\neq\emptyset$.
				 Let $\pi$ and $\pi'$
				 denote the projection from $(\bP^1_K)^n$
				 onto $(\bP^1)^{I_V}$ and $(\bP^1)^{I-I_V}$,
				 respectively. By Proposition \ref{product_of_curves},
				 $V=\pi_1(V)\times \pi_2(V)$
				 where $Z:=\pi_1(V)$ is a periodic
				 point of $(\bP^1)^{I_V}$. Write $W=\pi_2(V)$. 
				 By Step 1.1, the conclusion of Theorem \ref{intro_thm1} is valid for $W$.
				 Let $\varphi_1$ and $\varphi_2$ respectively denote
				 the diagonally split self-map of $(\bP^1)^{I_V}$
				 and $(\bP^1)^{I-I_V}$ associated to $f$.

				 \textbf{Step 1.2.1:} assume there 
				 is the largest $N$ such that
				 $\varphi_1^N(\pi_1(P))=Z$. Replace $P$ by $\varphi^{N+1}(P)$,
				 we can assume that the $\varphi_1$-orbit of
				 $\pi_1(P)$ does not contain $Z$. By Theorem 
					\ref{intro_thm0}, there exist infinitely many primes
					$\fp$ such that the $\fp$-adic closure $\cC_{\fp}$
					of $\cO_{\varphi_1}(\pi_1(P))$
					does not contain $Z$. For each such $\fp$,
					the $\fp$-adic closure of
				 $\cO_{\varphi}(P)$ is contained in
				 $\cC_{\fp}\times (\bP^1)^{I-I_V}(K_\fp)$
				 which is disjoint from $V(K_\fp)$.

				 \textbf{Step 1.2.2:} now assume $\varphi_1^n(\pi_1(P))=Z$
				 for infinitely many $n$. This implies that $Z$ is periodic
				 and $\pi_1(P)$ is preperiodic. Replacing $P$ by an iterate, we may assume
				 $\pi_1(P)=Z$. Let $N$ denote the exact period of $Z$. Since 
				 $\cO_{\varphi}(P)$ does not intersect $V$, we have that
				 $\cO_{\varphi_2^N}(\pi_2(P))$ does not
				 intersect $W$. By Step 1.1, the theorem holds for $W$.
				 Hence there exist infinitely many primes $\fp$ such that
				 the $\fp$-adic closure $\cC_{\fp}$
				 of $\cO_{\varphi_2^N}(\pi_2(P))$ does not
				 intersect $W(K_\fp)$. For each such $\fp$, the $\fp$-adic closure
				 of $\cO_{\varphi}(P)$ is contained in:
				 $$\left(\bigcup_{i=1}^{N-1} \left\{\varphi_1^{i}(Z)\right\}\times (\bP^1)^{I-I_V}(K_\fp)\right)
				 \cup \left\{Z\right\}\times \cC_{\fp}$$
				 which is disjoint from $V(K_\fp)$.
				 This finishes the case $V$ is periodic.

				  \textbf{Step 2:} assume $V$ is preperiodic and not periodic,
				   hence there exist $k> 0$
				   and $M>0$ such that $\varphi^{k+M}(V)= \varphi^{k}(V)$. For
				   $0\leq i<M$, write $V_i=\varphi^{k+i}(V)$.
				   Then we have that $V_i$ is periodic for every $0\leq i<M$.
				   
	
				   
				   \textbf{Step 2.1:} assume $I_V=\emptyset$. 
				   As in Step 1.1, we can relabel the factors 
				   of $(\bP^1)^n$ and rename the coordinate functions into
				   $x,y_1,\ldots,y_{n-1}$ so that
				   for each $0\leq i<M$, the periodic curve $V_i$
				   is given by equations $y_j=g_{i,j}(x)$ for
				   $1\leq j\leq n-1$,
				   where $g_{i,j}$ commutes with an iterate of $f$
				   and $\deg(g_{i,1})\leq\ldots\leq\deg(g_{i,n-1})$.
				   
					Since $V$ is not 
				   periodic, $V$ and $V_i$ are distinct curves,
				   hence $V\cap V_i$ is a finite set of points for every
				   $0\leq i<M$. By Lemma \ref{ext}, we extend
				   $K$ such that $\bP^1(K)$ contains the coordinates of all these
				   points. 
				   
				   Now we assume that for almost all
				   $\fp$, the $\fp$-adic closure of $\cO_{\varphi}(P)$
				   intersects $V(K_\fp)$ and we will arrive at a 
				   contradiction. Because $V_0=\varphi^k(V)$, for every such $\fp$,
				   the $\fp$-adic closure of $\cO_{\varphi}(P)$
				   intersects $V_0(K_\fp)$. Since
				   $V_0$ is periodic, the conclusion
				   of Theorem \ref{intro_thm1} has been established
				   for $V_0$. We must have that $V_0(K)$ contains
						an element in the orbit of $P$. By ignoring the
						first finitely many elements in that orbit, we may 
						assume $P\in V_0(K)$.
						Then we have $\varphi^{i+tM}(P)\in V_i(K)$
						for all $t\geq0$, $0\leq i<M$.  Let $a$
						denote the $x$-coordinate of $P$. For each $0\leq i<M$,
						let $n_i=|V\cap V_i|$, and let $u_{i,1},...,u_{i,n_i}$
						denote the $x$-coordinates of points in $V\cap V_i$.
						Since $V_i$ is defined by $y_j=g_i(x)$ for
						$1\leq j\leq n-1$, every point on $V_i$ is 
						uniquely determined by its $x$-coordinate. Since the orbit
						of $P$ does not intersect $V(K)$, we have:
						$$f^{i+tM}(a)\notin \{u_{i,1},\ldots,u_{i,n_i}\}\ 
						\forall t\geq 0,\ \forall 0\leq i<M.$$
		        Write $\cA=\displaystyle\bigcup_{0
		        \leq i<M}(f^i)^{-1}(\{u_{i,1},\ldots,u_{i,n_i}\})$. We have that
		        $f^{tM}(a)\notin \cA$
		        for all $t\geq 0$. By Theorem \ref{intro_thm0},
		        there exists infinitely many primes
		        $\fq$ such that the $\fq$-adic closure $C_\fq$
		        of $\{f^{tM}(a):\ t\geq 0\}$ 
		        does not intersect $\cA$. Now the $\fq$-adic closure
		        of the orbit of $P$ is contained in:
		        $$\bigcup_{0\leq i<M} \left(V_i(K_\fq)\cap \{(x,y_1,\ldots,y_{n-1})\in (\bP^1)^n(K_\fq):\ 
		        x\in f^i(C_\fq)\}\right)$$
		        which is disjoint from $V(K_\fq)$. This gives a contradiction and
		        finishes the 
		        case $I_V=\emptyset$.

		        \textbf{Step 2.2:} assume $I_V\neq\emptyset$. We can reduce to
		        Step 2.1 in exactly the same way we reduce
		        Step 1.2 to Step 1.1. This finishes the proof of
		        Theorem \ref{intro_thm1}.


				    
				   \subsection{Proof of Theorem \ref{intro_thm1*}}
				   In this subsection, we prove Theorem \ref{intro_thm1*}
				   by using induction on $n$. The cases $n\in\{1,2,3\}$
				   or $\dim(V)\in\{0,1,n-1\}$ have been established
				   by Theorem \ref{intro_thm0} and Theorem \ref{intro_thm1}
				   even without the extra technical assumption of Theorem \ref{intro_thm1*}.
				   Now 
				   assume $N>3$ and Theorem \ref{intro_thm1*} holds for all
				   $1\leq n<N$, we consider the case $n=N$. 
				   We may assume $\dim(V)>1$. As in the proof of Theorem
				   \ref{intro_thm1} in the last subsection, we 
				   can assume $I_V=\emptyset$ and deduce the
				   more general case in exactly the same way.

				    \textbf{Step 1:} assume $V$ is periodic. By Theorem \ref{inv_curves}, 
				    there exist
				    $1\leq i<j\leq n$ such that the image $\pi(V)$ of $V$ under the 
				    projection 
				    $$\pi\colon (\bP^1)^n\rightarrow(\bP^1)^2$$
				    onto the $(i,j)$-factor is a periodic curve.
				    We may assume $(i,j)=(1,2)$. If $\pi(\cO_{\varphi}(P))$
				    does not intersect $\pi(V)(K)$ then we can apply the induction
				    hypothesis. Otherwise, by ignoring the first finitely many
				    elements in the orbit of $P$, we may assume 
				    $\pi(P)\in \pi(V)(K)$. 
	%
	%
				    
				    Since $I_V=\emptyset$, we
				    may assume $\pi(V)$ is given by the equation $x_2=g(x_1)$ 
				    where 
				    $g$ commutes with an iterate of $f$ (the case $x_1=g(x_2)$ is 
				    similar).
				    Our technical assumption gives that $g$ commutes with $f$.
				    We consider the closed embedding:
				    $$(\bP^1)^{n-1}\stackrel{e}{\longrightarrow}(\bP^1)^{n}$$
				    defined by $e(y_1,\ldots,y_{n-1})=(y_1,g(y_1),y_2,\ldots,y_{n-1})$.
				    By pulling back under $e$, we reduce our problem
				    to the subvariety $(\bP^1)^{n-1}$ and apply the induction hypothesis.
				    This finishes the case $V$ is periodic.

				    \textbf{Step 2:} assume $V$ is preperiodic and not periodic. 
				    Write $\delta=\dim(V)$. 
				    By Proposition \ref{product_of_curves} and without
				    loss of generality, there exist
				    $m_0=0< m_1< m_2<\ldots<m_\delta=n$
				    such that $V=C_1\times C_2\times\ldots\times C_\delta$
				    where each $C_i$ is an $(f,\ldots,f)$-preperiodic
				    curve of $(\bP^1)^{m_i-m_{i-1}}$ for $1\leq i\leq \delta$.
				    For $1\leq i\leq \delta$, let $\pi_i$
				    denote the corresponding projection from $(\bP^1)^n$
				    onto $(\bP^1)^{m_i-m_{i-1}}$, and let $\varphi_i$
				    denote the self-map $(f,\ldots,f)$ of $(\bP^1)^{m_i-m_{i-1}}$.
				    If $P$ is preperiodic then there is nothing to prove, 
				    hence we may assume $P$ is wandering. Without loss of
				    generality, we assume $\pi_1(P)$
				    is $\varphi_1$-wandering.

				    \textbf{Step 2.1:} assume $C_1$ is not $\varphi$-periodic (recall
				    that it is preperiodic).
				    Then the set 
				    $$\bigcup_{j>0}C_1\cap \varphi_1^j(C_1)$$
				    is finite. Since $\pi_1(P)$
				    is wandering, there are only finitely many $j$'s such that
				    $\varphi_1^j(\pi_1(P))$ is contained in $C_1(K)$.
				    Ignore finitely many points in the orbits of $P$, we may
				    assume that the $\varphi_1$-orbit of
				    $\pi_1(P)$ does not intersect $C_1(K)$. Then we can apply the 
				    induction
				    hypothesis for the data
				    $((\bP^1)^{m_1},\varphi_1,\pi_1(P),C_1)$.
				    
				    \textbf{Step 2.2:} assume $C_1$
				    is $\varphi$-periodic. If the $\varphi_1$-orbit
				    of $\pi_1(P)$ does not intersect
				    $C_1(K)$
				    then we can apply the induction hypothesis as above. So we may
				    assume some element in this orbit is in $C_1(K)$. 
				    Replacing $P$ by an iterate, we may assume $\pi_1(P)\in C_1(K)$.
				    Since $I_V=\emptyset$, the curve $C_1$ is not contained in
				    any hypersurface of the form $x_j=\gamma$.
				    By Proposition \ref{product_of_curves} and the
				    discussion before it, we know that
				    $C_1$ is either $\bP^1$ if $m_1=1$ 
				    or is given by equations of the form (after possibly relabeling
				    the variables $x_1,\ldots,x_{m_1}$):
				    $x_2=g_1(x_1)$, $x_3=g_2(x_1)$, $\ldots$,
				    $x_{m_1}=g_{m_1-1}(x_{m_1-1})$, where 
				    each $g_j$ commutes with an iterate of $f$. 
				    By our technical assumption, every $g_j$ commutes with $f$. Hence
				    $C_1$ is $\varphi_1$-invariant, and we have
				    $\pi_1(\varphi^l(P))\in C_1(K)$
				    for every $l\geq 0$. Let $P'$ denote
				    the image of $P$ under the projection from $(\bP^1)^n$
				    to $(\bP^1)^{m_2-m_1}\times\ldots\times (\bP^1)^{m_\delta-m_{\delta-1}}$.
				    We now apply the induction hypothesis for the data:
				    $$((\bP^1)^{n-m_1},\varphi_2\times\ldots\times\varphi_\delta,
				    C_2\times\ldots\times C_\delta,P').$$ 
				    This finishes the proof of Theorem \ref{intro_thm1*}.

				 \subsection{Proof of
				 Theorem \ref{intro_thm3} when $V$ is a hypersurface}\label{notdis}
				 We first consider the case 
				 $\sigma\circ f\circ \sigma^{-1}=X^d$ for some 
				 $\sigma\in \Aut(\bP^1)$. By extending $K$, we may assume
				 $\sigma\in K[X]$. For almost all
				 $\fp$, $\sigma$ induces a homeomorphism
				 from $(\bP^1)^n(K_\fp)$ to itself. Hence we can assume
				 $f(X)=X^d$. Since the conclusion of
				 Theorem \ref{intro_thm3} is for almost all $\fp$,
				 we can assume $V$ is an absolutely irreducible 
				 preperiodic hypersurface defined over $K$. 
				 
				 First, assume there exists $1\leq i\leq 
				 n$ such that $V$ is given by $x_i=0$ or $x_i=\infty$.
					By the automorphism $X\mapsto X^{-1}$ and without
				 loss of generality, we may assume 
				 $V$ is given by $x_1=0$. Let $\alpha$ denote
				 the first coordinate of $P$, since the orbit of $P$ does not
				 intersect $V(K)$, we have $\alpha\neq 0$. For almost
				 all $\fp$, the $\fp$-adic closure of the orbit of $P$
				 lies in:
				 $$\{(x_1,\ldots,x_n)\in (\bP^1)^n(K_\fp):\ v_p(x_1)=0\}$$
				 which is disjoint from $V(K_\fp)$.

				 Therefore, we may assume $V\cap \Gmn\neq \emptyset$. It is
				 not difficult to prove that $V\cap \Gmn$
				 is a translate of a subgroup of codimension 1, see 
				 \cite[Remark 1.1.1]{Zan}. We now
				 denote the coordinate of each factor $\bP^1$ as
				 $x_1,...,x_q$, $y_1,...,y_r$ and $z_1,...,z_s$
				 (hence $q+r+s=n$)
				 such that $V$ is given by an equation:
				 		$$x_{1}^{a_1}...x_q^{a_q}=\zeta y_1^{b_1}....y_r^{b_r},$$
				 where $a_1,...,b_r$
				 are positive integers, and $\zeta$ is a 
				 root of unity. Actually, for $V$ to be preperiodic, 
				 we have $\displaystyle \zeta^{d^{A}}=\zeta^{d^{B}}$ for some $0\leq A<B$;
				 but we will not need this stronger fact. Write
				 $P$ under the corresponding coordinates as:
				 	$$P=(\alpha_1,...,\alpha_q,\beta_1,...,\beta_r,
				 	\gamma_1,...,\gamma_s).$$
				 Assume some elements among
				 the $\alpha_1,...,\beta_r$ are either 0 or $\infty$,
				 say, we have $\alpha_1=0$. Then each irreducible 
				 component of the intersection 
				 $V\cap \{x_1=0\}$ has the form
				 $\{x_1=0 \wedge x_i=\infty\}$
				 for $2\leq i\leq q$, or the form
				 $\{x_1=0 \wedge y_j=0\}$ for some
				 $1\leq j\leq r$. Thus the coordinates
				 of $P$ satisfy:
				 	$$(\alpha_i\neq \infty\ \forall 2\leq i\leq q)\wedge
				 		(\beta_j\neq 0\ \forall 1\leq j\leq r).$$
				 For every prime $\fp$, let $v_\fp(\infty)=-\infty$ (warning:
				 the $\infty$ on the left is an element of $\bP^1(K)$
				 while the $\infty$ on the right is 
				 an element of the extended real numbers). For almost
				 all primes $\fp$, the $\fp$-adic closure of
				 the orbit of $P$ is contained in:
				 $$\left\{(0,X_2,...,Z_s):\ v_{\fp}(X_i)\geq 0,\ 
				 v_{\fp}(Y_j)\leq 0\ \forall 2\leq i\leq q\ \forall 1\leq j
				 \leq r\right\}$$
				 which is disjoint from $V(K_\fp)$. The case, say, 
				 $\alpha_1=\infty$ is treated similarly.

				 Now we can assume that all the $\alpha_1,...,\beta_r$
				 lie in $\Gm(K)$. 
				  Let 
				 	$$\eta=\alpha_1^{a_1}...\alpha_q^{a_q}\beta_1^{-b_1}...
				 	\beta_r^{-b_r}.$$
				 Since the $\varphi$-orbit
				 of $P$ does not intersect 
				 $V(K)$, we have that the $f$-orbit of $\eta$
				 does not contain $\zeta$. For almost all 
				 $\fp$, we have $\eta$ is a $\fp$-adic unit.
				 By Theorem \ref{etale}, for almost
				 all $\fp$, the $\fp$-adic closure $C_{\fp}(\eta)$ of
				 the orbit of $\eta$ does not contain $\zeta$. Now
				 the $\fp$-adic closure of the orbit of $P$ lies in
				 $$\left\{(X_1,...,Z_s):\ X_1^{a_1}...Y_r^{-b_r}\in
				  C_{\fp}(\eta)\right\}$$	 
				 which is disjoint from $V(K_\fp)$. This finishes the case
				 $f$ is a conjugate of $X^d$.

				 Now we assume $f$ is a conjugate of $\pm C_d(X)$. As before,
				 we may assume $f(X)=\pm C_d(X)$. Let $\hat{f}=\pm X^d$,
				 and $\hat{\varphi}$ be the diagonally split morphism
				 corresponding $\hat{f}$. Consider the morphisms:
				 $$\Phi\colon (x_1,...,x_n)\mapsto
				 \left(x_1+\frac{1}{x_1},...,x_n+\frac{1}{x_n}\right)$$
				 from $(\bP^1)^n$ to itself. We have the commutative diagram:
				 $$\begin{diagram}
				 		\node{(\bP^1)^n}\arrow{s,l}{\Phi}\arrow{e,t}{\hat{\varphi}}
				 		\node{(\bP^1)^n}\arrow{s,r}{\Phi}\\
				 		\node{(\bP^1)^n}\arrow{e,b}{\varphi}\node{(\bP^1)^n}
				 \end{diagram}$$
				 
				 Extend $K$ further, we may assume there is $Q\in (\bP^1)^n(K)$
				 such that $\Phi(Q)=P$. Write $\hat{V}=\Phi^{-1}(V)$. 
				 We have that the $\Phi$-orbit
				 of $Q$ does not intersect $\hat{V}(K)$. Note that
				 the conclusion of the theorem has been established for $\hat{f}$. 
				 Therefore, for almost all 
				 $\fp$, the $\fp$-adic closure of
				 $\cO_{\hat{\varphi}}(Q)$ does not intersect $\hat{V}(K_\fp)$.
				 Since $\Phi$ is finite, it maps the 
				 $\fp$-adic closure of $\cO_{\hat{\varphi}}(Q)$ onto
				 the $\fp$-adic closure of $\cO_{\varphi}(P)$. We can conclude
				 that the $\fp$-adic closure of $\cO_{\varphi}(P)$
				 does not intersect $V(K_\fp)$.

				 \subsection{Proof of Theorem \ref{intro_thm3}} As in Subsection
				 \ref{notdis}, we first consider the case $f$ is conjugate
				 to $X^d$, and then we may assume $f(X)=X^d$. By Theorem \ref{intro_thm0}
				 and Subsection \ref{notdis}, we have that Theorem \ref{intro_thm3}
				 is valid when $n=1,2$. We proceed
				 by induction on $n$. Let $N\geq 3$ and assume that Theorem 
				 \ref{intro_thm3} holds
				 for all $n<N$, we now consider the case $n=N$.
				 As in Subsection \ref{notdis}, we assume $V$ is
					an absolutely irreducible
					preperiodic subvariety defined over $K$.  
					
					We first consider the case $V$ is contained in a hypersurface of
					the form $x_i=0$ or $x_i=\infty$ for some $1\leq i\leq n$.
					Without loss of generality, we may assume $V$ is
					contained in the hypersurface $x_1=0$. Let $\alpha$ denote
					the first coordinate of $P$. If $\alpha\neq 0$ then
					for almost all $\fp$, $\fp$-adic closure of the orbit of $P$
				  is contained in:
				  $$\{(x_1,\ldots,x_n):\ v_{\fp}(x_1)=0\}$$
				 which is disjoint from $V(K_\fp)$. Hence we assume $\alpha=0$. 
				 We now restrict to the hyperplane $x_1=0$ and apply the induction 
				 hypothesis.

				 Therefore we may assume $V\cap \Gmn\neq\emptyset$. 
				 Write $P=(\alpha_1,\ldots,\alpha_n)$. We first consider the case
				 $P\notin\Gmn$. Without loss of generality, assume $\alpha_1=0$.
				 We can again restrict to the hypersurface $x_1=0$
				 and apply the induction hypothesis.

				 Now consider the case $P\in \Gmn$. For almost all $\fp$, the
				 $\fp$-adic closure of the orbit of $P$ lies in:
				 $$(x_1,\ldots,x_n)\in(\bP^1)^n(K_\fp):\ v_p(x_i)=0\ \forall 1\leq i\leq n\}$$
				 which is closed in both $(\bP^1)^n(K_\fp)$
				 and $\Gmn(K_\fp)$. Hence it suffices to show
				 that for almost all $\fp$, the $\fp$-adic closure
				 of $\cO_{\varphi}(P)$ in $\Gmn(K_\fp)$
				 does not intersect $(V\cap\Gmn)(K_\fp)$.
				 This follows from the main result
				 of \cite[Theorem 4.3]{AKNTVV}.

				 \section{Dynamical Bombieri-Masser-Zannier Height Bound}\label{DynamicalBMZ}
				\subsection{Motivation and Main Results}
					This section presents the second arithmetic application of the 
					Medvedev-Scanlon 
					theorem to a dynamical analogue of ``complementary dimensional 
					intersections'' in $\Gmn$ 
					first studied by Bombieri-Masser-Zannier in \cite[Theorem 1]{BMZ99}.
					The story begins with the following:
					\begin{question}[Lang, Manin-Mumford]\label{LMM}
					Let $X$ be an abelian variety or the torus $\Gmn$ over $\C$.
					Let $C$ be an irreducible curve in $X$. Assume $C$
					is not a torsion translate of a subgroup. Is it true that
					there are only finitely many torsion points on $C$?
					\end{question} 
					
					This question has an affirmative answer. 
					When $X$ is an abelian variety, it is the Manin-Mumford conjecture first 
					proved
					by Raynaud \cite{Raynaud}. When $X=\Gmn$, it is a
					special case of a  question of
					Lang stated in the 1960s (see, for example, \cite{Lang-dio-book}) 
					which admits many proofs as well as generalizations.
					For example, Bombieri, Masser and Zannier \cite{BMZ99}
					obtain the following:
					\begin{theorem}[Bombieri, Masser, Zannier]\label{BMZ}
					Let $C$ be an irreducible curve in $\Gmn$ defined over
					a number field $K$ such that $C$ is not
					contained in any translate of a proper subgroup. Then
					\begin{itemize}
						\item [(a)] Points in $\displaystyle \bigcup_{V} 
							(C(\bar{K})\cap V(\bar{K}))$
							have bounded height, where $V$ ranges over
							all subgroup of codimension 1.
						\item [(b)] The set $\displaystyle \bigcup_{V}
							(C(\bar{K})\cap V(\bar{K}))$
							is finite, where $V$ ranges over all
							subgroups of codimension 2.
					\end{itemize}
					\end{theorem}

					While Question \ref{LMM}, and part (b)
					of Theorem \ref{BMZ} are instances of ``unlikely intersections''
					(see \cite{Zan}), 
					part (a) of
					Theorem \ref{BMZ} is an
					instance of ``not too likely intersections''. More precisely, we expect
					that the intersection appears infinitely many times, yet remains ``small''
					in a certain sense. 
					A conjectural dynamical analogue of Question \ref{LMM} 
					has been proposed by Zhang \cite{ZhangDist}
					and modified by Zhang, Ghioca, and Tucker \cite{GTZ}. However, we
					are not aware of any dynamical analogue of part (a) of Theorem
					\ref{BMZ}. By using canonical height arguments and 
					the Medvedev-Scanlon theorem, we obtain the following:

					

					\begin{theorem}\label{dynBMZ_1}
						Let $K$ be a number field 
						or a function field. 
						Let $f\in K[X]$ be a disintegrated polynomial, and $\varphi:\ 
						(\bP^1_K)^n\longrightarrow (\bP^1_K)^n$ be the corresponding
						split polynomial map. Let $C$ be an irreducible
						curve in $(\bP^1_{\bar{K}})^n$ that is
						not contained in any periodic hypersurface. Assume
						$C$ is non-vertical, by which we mean
						$C$ maps surjectively onto each factor $\bP^1$
						of $(\bP^1)^n$. Then 
						the points in 
							$$\bigcup_{V} (C(\bar{K})\cap V(\bar{K}))$$
							have bounded Weil heights, where $V$ ranges over 
							all periodic hypersurfaces of
							$(\bP^1_K)^n$.
					\end{theorem}

					We expect Theorem \ref{dynBMZ_1} still holds
					in the non-preperiodic case: $C$ is assumed to be
					not contained in any preperiodic hypersurface,
					and $V$ ranges over all preperiodic hypersurfaces. 
					However, we could only prove
					a bound on the ``average height'' of points
					in the intersections (see Theorem
					\ref{thm:C_cap_V}). In fact, such bound on the average height 
					turns out to hold 
					for a more general polarized dynamical system (see Theorem
					 \ref{thm:H_cap_V}). We prove this general result
					 by using various constructions of heights and canonical
					 heights
					 coming from
					 Gillet-Soul\'e generalization of Arakelov
					 intersection theory (see \cite{BGS}, \cite{Zhang95}, and
					 \cite{Kawa06}). At the end of this section, we also briefly explain why our
					 results continue to hold for the dynamics of 
					 split polynomial maps of the form $(f_1,\ldots,f_n)$, where
					$f_1,\ldots,f_n$ are disintegrated polynomials of degrees at least 2.
					This seemingly more general case is left to the end in order
					to make it easier for the readers to follow the main ideas, and more 
					importantly because this case can be easily
					reduced to the diagonally split case $(f,\ldots,f)$.

					Part (a) of Theorem \ref{BMZ} is only
					the beginning of a long and unfinished story. Subsequent papers
					by various authors
					have considered bounded height results for
					higher dimensional complementary intersections
					in the torus $\Gmn$ or an abelian variety.
					We refer the readers to \cite{BMZ07}, \cite{Hab_torus}, and
					\cite{Hab_abvar} as well as the references there for more details.
					As far as we know, the results given in this section
					are the first to indicate that the above results in diophantine 
					geometry are expected to hold, at least to some extent,
					in arithmetic dynamics. We will treat the dynamical analogue of
					higher dimensional complementary intersections
					in a future work. In this paper, we will be content with intersection
					between a curve and preperiodic hypersurfaces in $(\bP^1)^n$.
					  

					Throughout this section, let $f\in K[X]$ be a disintegrated 
					polynomial. We use $h$ to denote the absolute logarithmic Weil height 
					on $\bP^1(\bar{K})$. We
					also use $h$ to denote
					the height on $(\bP^1)^n(\bar{K})$
					defined by $h(a_1,\ldots,a_n)=h(a_1)+\ldots+h(a_n)$. 
					For every polynomial $P\in \bar{K}[X]$
					of degree at least 2, we let $\hat{h}_P$
					denote the canonical height associated to $P$. We use
					$\hat{h}$ to denote the canonical height 
					$\hat{h}_f$. For
					properties of all these height functions, see \cite[Part B]{HS}
					and  \cite[Chapter 3]{Sil-ArithDS}.
					
					\subsection{Proof of Theorem \ref{dynBMZ_1}}
					 Since the projection from $C$ to each
					 factor $\bP^1$ is finite, to show that a collection
					 of points in $C(\bar{K})$ has bounded heights, it
					 suffices to show that for some $1\leq i\leq n$,
					 all their $x_i$-coordinates have bounded heights.
					 By the Medvedev-Scanlon Theorem, it suffices
					 to show that for every $1\leq i<j\leq n$, points
					 in $\bigcup C(\bar{K})\cap V_{ij}(\bar{K})$
					 have bounded heights where $V_{ij}$ ranges over
					 all periodic hypersufaces whose equation involving 
					 $x_i$ and $x_j$ only. Therefore we may assume $n=2$
					 for the rest of this subsection.
					 Let $x$ and $y$ denote the
					 coordinate functions on the first and second factor $\bP^1$
					 respectively. Without loss of generality, we only need to
					 consider the intersection with periodic curves $V$ given by an 
					 equation
					 of the form $x=\zeta$ where $\zeta$
					 is $f$-periodic, or $y=g(x)$ where  $g$ commutes
					 with an iterate of $f$. Now every periodic $\zeta$ has height 
					 bounded
					 uniformly, we get the desired conclusion
					 when intersecting $C$ with curves of the form $x=\zeta$. Note
					 that this argument also works for all
					 preperiodic $\zeta$. 
					 
					 So we only have to consider curves $V$ of the form 
					 $y=g(x)$. Let $(M,N)$ denote the type
					 of the divisor $C$ of $(\bP^1)^2$. Explicitly,
					 we choose a generator $F(x,y)$ 
					 of the (prime) ideal of $C$ in  $\bar{K}[x,y]$,
					 then $F$ has degree $M$ in $x$ and degree $N$ in $y$.
					 We have the following two easy lemmas:
					 
					 \begin{lemma}\label{ineq_C}
					 For every point $(\alpha,\beta)$ in $C(\bar{K})$,
					 we have:
					 \begin{equation}\label{c1_c2}
					 	|M\hat{h}(\alpha)-N\hat{h}(\beta)|\leq c_1\sqrt{\hat{h}(\alpha)+\hat{h}(\beta)+1}+c_2.
					 \end{equation}
					 where $c_1$ and $c_2$ are constants independent of $(\alpha,\beta)$.
					 \end{lemma}
					 \begin{proof}
					 	Let $\tilde{C}$ denote the normalization of $C$, we have:
			\begin{equation}\label{data_C}
			\tilde{C}\stackrel{\eta}{\longrightarrow}C\stackrel{i}
			{\longrightarrow}(\bP^1_K)^2
			\end{equation}
			where $\eta$ is the normalization map and $i$ is
			the closed embedding realizing $C$ as a subvariety of $(\bP^1_K)^2$.The invertible sheaf $\scrL:=(\eta\circ i)^{*}\scrO(1,1)$
			is ample on $\tilde{C}$.
			Let $\pi_1$ and $\pi_2$ denote respectively the first
			and second projections from $(\bP^1_K)^2$ to $\bP^1_K$. The 
			invertible sheaves
			$\scrL_1:=(\eta\circ i\circ\pi_1)^{*}\scrO(1)$ and
			$\scrL_2:=(\eta\circ i\circ\pi_2)^{*}\scrO(1)$ have degrees
			$N$ and $M$, respectively. 
			
			For $j=1,2$, define $\tilde{h}_j(P)=h(\pi_j\circ i\circ\eta (P))$ 
			for every
			$P\in \tilde{C}(K)$. We also define
			$\tilde{h}(P)=h(i\circ\eta(P))$ for every $P\in \tilde{C}(K)$. Then
			$\tilde{h}$, $\tilde{h}_1$ and $\tilde{h}_2$ respectively
			are  height functions on $\tilde{C}(K)$ corresponding
			$\scrL$, $\scrL_1$ and $\scrL_2$. By \cite[Theorem B.5.9]{HS},
			there is a constant $c_1>0$ depending only on 
			the data (\ref{data_C}) such that:
			\begin{equation}\label{ineq:1}
				|M\tilde{h}_1(P)-N\tilde{h}_2(P)|\leq c_1\sqrt{\tilde{h}(P)+1}\ \ 
				\forall P\in \tilde{C}(K)
				\end{equation}
		
			For every point $(\alpha,\beta)\in C(K)$, inequality (\ref{ineq:1})
			gives:
			\begin{equation}\label{ineq:2}
				|Mh(\alpha)-Nh(\beta)| \leq c_1\sqrt{h(\alpha)+h(\beta)+1}.
			\end{equation}
			In term of the canonical height function
			associated to $f$, inequality (\ref{ineq:2})
			becomes:
			\begin{equation}\label{ineq:3}
				|M\hat{h}(\alpha)-N\hat{h}(\beta)|\leq c_1\sqrt{\hat{h}(\alpha)
				+\hat{h}(\beta)+1}+c_2
			\end{equation}
			where $c_2$ only depends
			on $f$ and the data (\ref{data_C}).
					 	\end{proof}
					 
%
	
	\begin{lemma}\label{same_cano}
	Let $P\in\bar{K}[X]$ be a disintegrated polynomial, $\scrG$
	a finite cyclic subgroup of linear polynomials in $\bar{K}[X]$
	such that for some positive integer $D$, we have $P\circ L=L^D\circ P$ for every $L\in\scrG$. We have:
	\begin{itemize}
	\item [(a)] $\displaystyle \hat{h}_{P}=\hat{h}_{L\circ P^l}$
	for every $l>0$ and every $L\in \scrG$.   
	\item [(b)] $\displaystyle 
	\hat{h}_{P}(L(\alpha))=\hat{h}_{P}(\alpha)$ for every $L\in \scrG$ 
	and $\alpha\in\bP^1(\bar{K})$.
	\end{itemize}
	\end{lemma}				
	\begin{proof}
		Since $\scrG$ is finite, there is $\epsilon$ such that: 
		$$h(x)-\epsilon\leq h(L(x))\leq h(x)+\epsilon\ \forall x\in \bar{K}\ \forall L\in \scrG.$$
		For every $k\geq 1$, we have $(L\circ P^l)^k=\tilde{L}\circ P^{kl}$
		for some $\tilde{L}\in \scrG$. And we have:
			$$h((L\circ P^l)^k(x))=h(P^{kl}(x))+O(1)$$
			where $O(1)$ is bounded
			independently of $k$. Dividing both sides by
			$\deg(P^{kl})$ and let $k\to\infty$ will kill off this $O(1)$.
			Part (b) is proved similarly.
	\end{proof}

	We can now finish the proof of Theorem \ref{dynBMZ_1}.	Let
	$V$ be given by $y=g(x)$ and 
	$(\alpha,\beta)$ be a point in the intersection $C\cap V$. By Lemma
	\ref{same_cano}, we have
	$\hat{h}(\beta)=\deg(g)\hat{h}(\alpha)$. Substituting this
	into (\ref{c1_c2}), we have:
	\begin{equation}\label{main_ineq}
		|M-N\deg(g)|\hat{h}(\alpha)\leq c_1\sqrt{(\deg(g)+1)\hat{h}(\alpha)+1}+c_2
	\end{equation}

	For all sufficiently large $\deg(g)$ (for instance, we may choose
	$\displaystyle\deg(g)>\frac{2M}{N}$
	so that $\displaystyle N\deg(g)-M>\frac{N\deg(g)}{2}$), inequality
	(\ref{main_ineq}) implies
	that $\hat{h}(\alpha)$ and hence $h(\alpha)$ is bounded above
	by a constant depending only
	on $f$ and the data (\ref{data_C}). Therefore by
	the remark at the first paragraph of
	this subsection, $h(\alpha,\beta)$ is bounded
	by a constant depending only on $f$ and the data
	(\ref{data_C}). Finally, by 
	Proposition \ref{prop:all_g}, there
	are only finitely many such $g's$ of bounded degree, hence
	only finitely many points in the intersection
	$C\cap \{y=g(x)\}$. This finishes
	the proof of Theorem \ref{dynBMZ_1}.
	  
	
	
	\subsection{Further Questions}
	We now gather several questions concerning the union $\displaystyle 
	\bigcup_{V}(C(\bar{K})\cap V(\bar{K}))$ where $V$ ranges over preperiodic
	hypersurfaces in $(\bP^1_{\bar{K}})^n$ and $C$ is not contained
	in any such hypersurface. For each $k\geq 0$, let 
	$\scrP_k$ denote the collection of all hypersurfaces $V$
	of $(\bP^1_{\bar{K}})^n$ such that $\varphi^k(V)$
	is periodic. Thus $\scrP_0$ is exactly
	the collection of periodic hypersurfaces, and we have
	$\scrP_{k}\subseteq\scrP_{k+1}$ for every $k$.
	Apply Theorem 
	\ref{dynBMZ_1}
	for $\varphi^k(C)$, let $\Gamma_k$
	denote an upper bound for the $f$-canonical heights of 
	points in 
	$$\bigcup_{V\in\scrP_0}(\varphi^k(C)(\bar{K})\cap V(\bar{K})).$$
	Using
	$$\varphi^{k}(\bigcup_{V\in\scrP_k}(C(\bar{K})\cap V(\bar{K})))
		\subseteq \bigcup_{V\in\scrP_0}(\varphi^k(C)(\bar{K})\cap V(\bar{K})),$$
	we have that points in $\displaystyle \bigcup_{V\in\scrP_k}(C(\bar{K})\cap V(\bar{K}))$
	have canonical heights bounded by $\displaystyle \frac{\Gamma_k}{d^k}$
	where $d\geq 2$ is the degree of $f$. Heuristically speaking,
	suppose we could obtain a bound in Theorem \ref{dynBMZ_1}
	that depends, in a uniform way, on the ``complexity'' of $C$, and the 
	``complexity''
	of $\varphi^k(C)$ is ``essentially'' the ``complexity'' of $C$ 
	multiplied 
	by $d^k$. Then we have that $\Gamma_k=d^k O(1)$
	where $O(1)$ is independent of $k$. All of these motivate
	the following questions. From now on, we
	assume $K$ is a number field although
	the first two questions could be asked for
	function fields as well:  
	\begin{question}\label{Q}
	Let $f$ and $\varphi$ be as in Theorem \ref{dynBMZ_1}.
	\begin{itemize}
		\item [(a)] Let $C$ be an irreducible non-vertical 
		curve in $(\bP^1_{\bar{K}})^n$. Suppose $C$ is not contained
		in an element of $\scrP_k$. Is it true
		that points in
				$$\bigcup_{V\in\scrP_k}(C(\bar{K})\cap V(\bar{K}))$$
				have heights bounded independently of $k$.
		\item [(b)] Let $C$ be an irreducible non-vertical 
		curve in $(\bP^1_{\bar{K}})^n$ that is not contained
		in any preperiodic hypersurface. 
		Is it true
		that points in
				$$\bigcup_{V}(C(\bar{K})\cap V(\bar{K}))$$
				have bounded heights, where $V$ ranges
				over all preperiodic hypersurfaces of $(\bP^1_{\bar{K}})^n$?
		\item [(c)] Let $C$ be as in part (b).
		Is it true that the union
		in (b) have only finitely points of bounded degree?
		\item [(d)] Let $C$ be as in (b). Assume $C$ is
		defined over $K$. Is it true that the union in (b)
		have only finitely many $K$-rational points?
	\end{itemize}
	\end{question}
	
	It is obvious that these questions have decreasing strength. We now focus on Question \ref{Q}(b). We look more closely
	to the proof of Theorem \ref{dynBMZ_1} and see what still go
	through. Assume $f$, $\varphi$ and $C$ as in part (b) of Question
	\ref{Q}. As before, we can assume  
	$V$ ranges over all irreducible 
	preperiodic hypersurfaces. Let $k\geq0$
	such that $\varphi^k(V)$ is periodic, hence given
	by an equation of the form, say, $x_j=g(x_i)$ where
	$1\leq i<j\leq n$
	(the case $\varphi^k(V)$ is given by $x_i=\zeta$ where $\zeta$
	is preperiodic is easy). We can now assume $n=2$ by projecting to the 
	$(i,j)$-factor $(\bP^1)^2$ of $(\bP^1)^n$. Let $(\alpha,\beta)\in
	C(\bar{K})\cap V(\bar{K})$. From
	$f^k(\alpha)=g(f^k(\beta))$ and
	Lemma \ref{same_cano}, we still have
	$\hat{h}(\alpha)=\deg(g)\hat{\beta}$. Therefore inequality 
	(\ref{main_ineq}) still holds. We still have that
	$h(\alpha,\beta)$ is bounded when $\deg(g)$ is sufficiently large.
	Since there are only finitely many $g$'s of bounded degrees (see
	Proposition \ref{prop:all_g}), one
	may assume that the periodic hypersurface $\{x_j=g(x_i)\}$ is fixed. 
	Our discussions so far implies that Question
	\ref{Q}(b) is equivalent to the following:
	
	\begin{question}\label{fix_g}
	Let $f$, $\varphi$ and $C$ be as in Question \ref{Q}(b). Let $W$
	be a fixed irreducible periodic hypersurface of $(\bP^1_{\bar{K}})^n$.
	For $k\geq 0$, write $\varphi^{-k}(W)$
	to denote $(\varphi^{k})^{-1}(W)$. Is it true that 
	points in $\bigcup_{k\geq 0} C(\bar{K})\cap \varphi^{-k}(W)(\bar{K})$
	have bounded heights?
	\end{question}
	
	We  now focus on Question \ref{fix_g}. We could only prove a weaker result, namely points in 
	$C(\bar{K})\cap \varphi^{-k}(W)(\bar{K})$ have bounded ``average heights'' independent of $k$ (see Subsection \ref{sub:average}). Such a
	result is motivated by examples given in the next subsection.
	
	\subsection{Examples}
	Let $f$, $\varphi$, $W$ and $C$ be as in Question \ref{fix_g}.
	We may assume $n=2$ and $W$ is given by $y=g(x)$ where $g$ commutes
	with an iterate of $f$. In this subsection, we look at the case when
	$C$ is a rational curve parametrized by $(P(t),Q(t))$ where
	$P$ and $Q$ are polynomials with coefficients in $\bar{K}$. Question
	\ref{fix_g} asks whether roots of $f^k\circ Q=g\circ f^k\circ P$
	have heights bounded independently of $k$. 
	Note that if $\alpha$ is such a root
	then $\hat{h}(Q(\alpha))=\deg(g)\hat{h}(P(\alpha))$
	by Lemma \ref{same_cano}. Therefore
	$|\deg(Q)-\deg(g)\deg(P)|h(\alpha)$ is bounded independently
	of $k$. Hence if $\deg(g)\deg(P)\neq \deg(Q)$ then
	Question \ref{fix_g} has an affirmative answer. For the rest 
	of this subsection, we may assume
	$\deg(g)\deg(P)=\deg(Q)$.
	
	Since $g$ commutes with an iterate $f^l$ of $f$,
	we may look at $l$ collections of equations of the form
		$$f^{ql+r}\circ Q = g\circ f^{ql+r}\circ P = f^{ql}\circ 
		g\circ f^{r}\circ P\ \ \text{for $0\leq q$,}$$
	for each $0\leq r<l$. Replacing $(P,Q)$ by
	$(g\circ f^r\circ P,f^r\circ Q)$, we may
	assume $g(x)=x$ (i.e. $V_0$ is the diagonal),
	 and hence $\deg(P)=\deg(Q)$.
	For every $k\geq 0$, put $G_k=f^k\circ P-f^k\circ Q$. We need
	to show that roots of $G_k$ have heights bounded independently of $k$.
	By making a linear change, we can assume $f$ has
	the following form:
	$$f(X)=X^d+a_{d-2}X^{d-2}+\ldots+a_0.$$
	We have the following:
	\begin{lemma}\label{zeta_12}
		The linear automorphism of the Julia set of $f$ is a cyclic group
		$\mathscr{G}(f)$ of order $M$. Let $D\in \Z$ 
		such that $f\circ L=L^D\circ f$ for every $L\in \mathscr{G}(f)$. 
		If $\zeta_1$ and $\zeta_2$ are 
		two roots of unity such that $f(\zeta_1 X)=\zeta_2 f(X)$
		then $\zeta_1\in \mathscr{G}(f)$, and $\zeta_2=\zeta_1^D$.
	\end{lemma}
	\begin{proof}
	This is a classical result in complex dynamics, see \cite{Beardon}
	and \cite{SchSte}.
	\end{proof}
	
	\begin{proposition}\label{prop:average}
		\begin{itemize}
			\item [(a)] There is a constant $c_3$ such that
			$$\deg(G_k)\geq c_3d^k\ \ \forall k.$$
			\item [(b)] There is a constant $c_4$
			such that the affine height of $G_k$ (\cite[Part B]{HS}) satisfies:
					$$h(G_k)\leq c_4 d^k\ \ \forall k.$$
			\item [(c)] The average height of the roots of $G_k$
			are bounded independently of $k$: there is $c_5$ such that:
					$$\frac{1}{\deg(G_k)}\displaystyle\sum_{G_k(\alpha)=0} h(\alpha)\leq c_5
					\ \ \forall k,$$
					where we allow repeated roots to appear multiple times
					in $\sum$. 
		\end{itemize}
	\end{proposition}
	\begin{proof}
		For part (a): if there are $k_1<k_2$ and roots of unity
		$\zeta_1,\zeta_2$ such that: 
		$$f^{k_i}\circ Q=\zeta_i f^{k_i}\circ P\ \text{for $i=1,2$,}$$
		 then we have that $$f^{k_2-k_1}(\zeta_1 X)=\zeta_2 f^{k_2-k_1}(X).$$
		 By Lemma \ref{zeta_12}, we have that $\zeta_1\in\mathscr{G}(f)$.
		 Hence the curve $\varphi^{k_1}(C)$ which is given
		 by $y=\zeta_1 x$ is
		 preperiodic, contradiction. Therefore, there exists
		 $\mu$ such that $\displaystyle \frac{{f^k\circ Q}}{f^k\circ P}$
		 is not a root of unity for $k\geq\mu$. 
		 Write:
		 $$f^{k-\mu}(X)=X^{d^{k-\mu}}+b_{d^{k-\mu}-2}X^{d^{k-\mu}-2}+...$$
		 We have (this is the only place where we do not follow our convention on notation: $N$ in the exponent means the usual
		 ``raising to the $N^{\text{th}}$ power''
		 instead of ``taking the $N^{\text{th}}$ iterate''):
		 $$(f^{\mu}\circ Q)^N-(f^{\mu}\circ 
		 P)^N=\prod_{\zeta^N=1}(f^{\mu}\circ Q-\zeta f^{\mu}\circ P).$$
		 Since at most one factor has
		 degree lower than $d^{\mu}\deg(P)$, and that factor
		 is a nonzero polynomial, we have:
		 $$\deg((f^{\mu}\circ Q)^N-(f^{\mu}\circ 
		 P)^N)\geq (N-1)d^{\mu}\deg(P).$$
		 Therefore:
		 $$\deg(G_k)=\deg(f^{k-\mu}\circ f^\mu\circ Q-f^{k-\mu}\circ f^\mu\circ Q)\geq (d^{k-\mu}-1)d^{\mu}\deg(P).$$
		 This finishes part (a).
		 
		 For part (b), it suffices to show there are constants
		 $\epsilon_1$ and $\epsilon_2$ such that $h(f^k\circ P)\leq \epsilon_1 d^{k}$ 
		 and $h(f^k\circ Q)\leq \epsilon_2 d^{k}$ for every $k$. By similarity, we only
		 need to prove the existence of $\epsilon_1$.
		 Let $r_1,...,r_{d^k}$ denote the roots of $f^k$. Since
		 $\displaystyle \hat{h}_{f}(r_i)=\frac{\hat{h}_{f}(0)}{d^k}$,
		 we have:
		 \begin{equation}\label{hri}
		 	\sum_{i=1}^{d^k} h(r_i)=\hat{h}_{f}(0)+O(1)d^k
		 \end{equation}
		 where $O(1)$ only depends on $f$ (since we change
		 from canonical height to Weil height).
		 From $\displaystyle f^k\circ P=\prod_{i=1}^{d^k}(P-r_i)$,
		 and \cite[Proposition B.7.2]{HS}
		 we have:
		 \[\begin{array}{lcl}
		 	h(f^k\circ P)&\leq &\displaystyle\sum_{i=1}^{d^k}(h(P-r_i)+(\deg(P)+1)\log 2)\\
		 							&\leq &\displaystyle\sum_{i=1}^{d^k}(h(P)+h(r_i)+
		 							(\deg(P)+2)\log 2)\\
		 							&=&\hat{h}_{f}(0)+d^k(h(P)+(\deg(P)+2)\log 2 + O(1))
		 \end{array}\]
		 where the last equality follows from (\ref{hri}), so
		 the error term $O(1)$ only depends on $f$. Finally
		 part (c) follows from part (a), part (b) and \cite[Theorem 1.6.13]{BG}.
	\end{proof}


	Part (c) of Proposition \ref{prop:average} only gives us an upper bound
	(independent of $k$)
	for the average of the heights of roots of $G_k$ instead of 
	the height of every root. Now suppose
	there is a constant $c_6$ (independent of $k$) such that
	for every $k$, every root $\alpha$ of $G_k$
	that is not a root of $G_{k-1}$ has
	degree at least $c_6 d^k$ over $K$ then we are done. The reason
	is that there are at least $c_6d^k$
	conjugates of $\alpha$ and all contribute
	the same height to the average. It is usually the case 
	in the dynamics of disintegrated $f$ that every irreducible factor (in $K[X]$)
	of $G_k$ has a large degree unless it has already been a factor of 
	$G_{k-1}$. However, while such phenomena appear in practice, it seems to be a 
	very difficult problem to prove that such lower bounds on the degrees
	hold in general.
	We conclude this subsection by cooking up a specific instance
	in which all irreducible factors of $\displaystyle \frac{G_k}{G_{k-1}}$
	have large degrees thanks to the Eisenstein's criterion.
	\begin{proposition}\label{prop:example}
	Let $d\geq 2$ and let $p>d$ be a prime. 
	Let $f(X)=X^d+p$, and $C$ be the curve
	$y=x+p$ in $(\bP^1_{\Q})^2$. Then $C$
	is non-preperiodic and points in $\bigcup_{V}C(\bar{K})\cap V(\bar{K})$
	have bounded heights, where $V$ ranges over all
	preperiodic curves of $(\bP^1_{\Q})^2$. 
	\end{proposition}
	\begin{proof}
		By Theorem \ref{inv_curves}
		and Proposition \ref{prop:all_g}, non-preperiodicity of $C$ is equivalent 
		to
		$f^k(X)\neq \zeta f^k(X+p)$ for every $k$, 
	  and
		this is obvious. Hence $C$ is non-preperiodic.

		We have:
		$$G_k=f^k(x+p)-f^k(x)=
		\displaystyle\prod_{\zeta^d=1}(f^{k-1}(x+p)-\zeta f^{k-1}(x)).$$
		By the reduction from part (b) of Question \ref{Q}
		to Question \ref{fix_g}, it suffices to show
		that for every periodic $W$, points in
		$\bigcup_{k\geq 0}C(\bar{K})\cap \varphi^{-k}(W)(\bar{K})$
		have bounded heights. By the argument in the beginning of
		this subsection, we may assume $W$ is the diagonal. Hence it
		suffices to show that roots of $G_k$ have bounded heights
		independent of $k$. By Eisenstein's criterion, $f^{k-1}(x+p)-\zeta f^{k-1}(x)$ is irreducible
		(over $\Q(\zeta)$) when $\zeta\neq 1$. Then by Proposition \ref{prop:average} and the discussion after it, we get the
		desired conclusion.	 
	\end{proof}
	
	\subsection{The Bounded Average Height Theorem}\label{sub:average}
	\subsubsection{The Statements}
	In this subsection, we prove that the average bounded height result
	in Proposition \ref{prop:average} holds for an arbitrary
	polarized dynamical system (see Theorem \ref{thm:H_cap_V}).
	We have the following:
	\begin{theorem}\label{thm:C_cap_V}
		Let $f$, $n$ and $\varphi$ be as in Theorem \ref{dynBMZ_1}. Let $C$ be an
		irreducible curve in $(\bP^1_{\bar{K}})^n$ such that its projection 
		to each factor $\bP^1_{\bar{K}}$ is onto. There exists a constant $c_7$
		such that for every irreducible preperiodic hypersurface $V$
		in $(\bP^1_{\bar{K}})^n$ that does not contain $C$, the average height
		of points in 
		$C(\bar{K})\cap V(\bar{K})$ is bounded above by $c_7$. More precisely,
		define:
		$$C_{\bar{K}}.V=m_1P_1+\ldots+m_lP_l$$
		where $C(\bar{K})\cap V(\bar{K})=\{P_1,\ldots,P_l\}$
		and $m_1,\ldots,m_l$ are the corresponding intersection multiplicities. Then
		we have:
		\begin{equation}\label{ineq:C_cap_V}
	 		\displaystyle\frac{\sum_{i=1}^{l}m_ih(P_i)}
	 		{\sum_{i=1}^{l}m_i}\leq c_7.
	 		\end{equation}
	\end{theorem}
	
	As in the proof of Theorem \ref{dynBMZ_1}, we can simply reduce to the case
	$n=2$. Then Theorem \ref{thm:C_cap_V} is a special case of the following:
	\begin{theorem}\label{thm:H_cap_V}
		Let $X$ be a projective scheme over $K$ such that
		$X_{\bar{K}}$ is normal and irreducible, $H$
		a closed subscheme of $X$ such that $H_{\bar{K}}$ is an irreducible
		hypersurface. 
		Assume the line bundle $L$ associated to $H$ is
		very ample.
		Let $d\geq 2$, and let $\varphi$ be a $K$-morphism from $X$ to itself such 
		that 
		$\varphi^{*}L\cong L^d$. Fix a height $\tilde{h}$
		on $X(\bar{K})$ corresponding to a very ample line bundle. 
		There exists $c_8$ such that 
		for every irreducible $\varphi$-preperiodic
		curve $V$ of $X_{\bar{K}}$ not contained in $H_{\bar{K}}$, the average height
		of points in $H(\bar{K})\cap 
		V(\bar{K})$ is bounded above by $c_8$. More precisely, write: 
		$$H_{\bar{K}}.V_{\bar{K}}=m_1P_1+\ldots+m_2P_2$$
		where $H(\bar{K})\cap V(\bar{K})=\{P_1,\ldots,P_l\}$
		and $m_1,\ldots,m_l$ are the corresponding multiplicities. Then we have:
		\begin{equation}\label{ineq:H_cap_V}
	 		\displaystyle\frac{\sum_{i=1}^{l}m_i\tilde{h}(P_i)}
	 		{\sum_{i=1}^{l}m_i}\leq c_8.
	 		\end{equation}
	\end{theorem}
		
		We now focus on proving Theorem \ref{thm:H_cap_V}. Note the
		amusing change that we now concentrate on the intersection of
		a fixed \textit{hypersurface} with an arbitrary preperiodic \textit{curve}.
		We regard $X$ as a closed subvariety of $\bP^N_K$ by choosing
		a closed embedding associated to $H$. Let $h$ denote
		the Weil height on $\bP^N(\bar{K})$ as well as its restriction
		on $X(\bar{K})$. We may prove Theorem \ref{thm:H_cap_V} with
		$\tilde{h}$ replaced by $h$ since there exists $M$ such that
		$\tilde{h}< Mh + O(1)$ where the error term $O(1)$ is uniform on 
		$X(\bar{K})$. The main ingredients of the proof of Theorem \ref{thm:H_cap_V}
		are the arithmetic B\'ezout's theorem by Bost-Gillet-Soul\'e \cite{BGS},
		and the construction of the canonical height for subvarieties by Zhang
		\cite{Zhang95} and Kawaguchi \cite{Kawa06}.
	
	\subsubsection{Proof of Theorem \ref{thm:H_cap_V}} Let $V$ be a $\varphi$-preperiodic
	curve in $X_{\bar{K}}$. Let $F$ be a finite extension of $K$ such that $V$ is defined over
	$F$. Write $\scrO=\scrO_K$ to denote the ring of integers of $K$,
	and $\pi$ to denote the base change morphism $\bP^N_{F}\longrightarrow \bP^N_K$. 
	As in \cite[p.946--947]{BGS}, we 
	let $\bar{E}$ denote the
	trivial hermitian vector bundle of rank $N+1$ on $\Spec(\scrO)$ and
	equip the canonical line bundle $\scrM:=\scrO(1)$ of $\bP^N_{\scrO}$
	with the quotient metric $m$. We denote $\bar{\scrM}=(\scrM,m)$. The 
	pull-back 
	of $\scrM$ to $X$ is isomorphic to the line bundle $L$.

	For $0\leq p\leq N+1$, for any cycle $\scrZ\in Z_{p}(\bP^N_{\scrO})$
	of dimension $p$, following
	\cite[p. 946]{BGS}, we define the
	Faltings' height of $\scrZ$ to be the real number:
	\begin{equation}\label{F_height_O}
		h_{Fal}(\scrZ)=\widehat{\deg}\left(\hat{c}_1(\bar{\scrM})^{p}\mid \scrZ\right)
	\end{equation} 	
	where $\hat{c}_1(\bar{\scrM})$ is
	the first arithmetic Chern class of $\bar{\scrM}$,
	and $\widehat{\deg}$ is the arithmetic degree map as defined in 
	\cite{BGS}. 
	
	For $0\leq p\leq N$, for every cycle $Z\in Z_p(\bP^N_{K})$,
	let $\bar{Z}$ denote the closure of
	$Z$ in $\bP^N_{\scrO}$. We define the Faltings' height 
	of $Z$ to be:
	\begin{equation}\label{F_height}
		h_{Fal}(Z):=h_{Fal}(\bar{Z}).
	\end{equation}
	If $Z\in Z_p(\bP^N_{\bar{K}})$, we let $K'$ be a finite extension of $K$
	so that $Z$ is defined over $K'$, i.e. $Z$ is the pull-back of
	a cycle $Z'\in Z_p(\bP^N_{K'})$. Let $\rho$ denote
	the base change morphism from $\bP^N_{K'}$
	to $\bP^N_K$. We then define the Faltings' height of $Z$
	to be:
	\begin{equation}\label{F_height_Kbar}
	h_{Fal}(Z):=\frac{1}{[K':K]}h_{Fal}(\rho_*Z')
	\end{equation}
	This is independent of the choice of $K'$.

	 	For $0\leq p\leq N+1$, for every 
	 	cycle $\scrZ\in Z_{p}(\bP^N_{\scrO})$, we have the following 
	Bost-Gillet-Soul\'e 
	projective
	height of $\scrZ$ \cite[p. 964]{BGS}:
	\begin{equation}\label{BGS_height_O}
		h_{BGS}(\scrZ)=\widehat{\deg}(\hat{c}_p(\bar{Q})\mid \scrZ)
	\end{equation}
	where $\bar{Q}$ is the hermitian vector bundle defined as in \cite[p. 964]{BGS}, and $\hat{c}_p$
	is the $p^\text{th}$ arithmetic Chern class of $\bar{Q}$.
	 	
	 	For $0\leq p\leq N$, for every cycle $Z\in Z_p(\bP^N_{K})$,
	 	we define the Bost-Gillet-Soul\'e height of $Z$ to be:
	 	\begin{equation}\label{BGS_height}
		h_{BGS}(Z):= h_{BGS}(\bar{Z}).
	  \end{equation}
	 	If $Z\in Z_p(\bP^N_{\bar{K}})$, we let $K'$ be a finite extension of $K$
	 	over which $Z$ is defined by $Z'\in Z_p(\bP^N_{K'})$. Let $\rho$ be 
	 	the base change morphism as above, we define:
	 	\begin{equation}\label{h_BGS_Kbar}
	 	h_{BGS}(Z):=\frac{1}{[K':K]}h_{BGS}(\rho_*Z').
	 	\end{equation}
	 This is independent of the choice of $K'$.

	 	Proposition 4.1.2 in \cite{BGS} in which the authors compare
	the Faltings' height and the Bost-Gillet-Soul\'e projective
	height yields the following:
	\begin{proposition}\label{prop:h_F_h_BGS}
		For $0\leq p\leq N$, for any cycle $Z\in Z_{p}(\bP^N_{K})$, 
		define
		$\deg_K(Z)=\deg_K(\bar{Z}):=\deg_{\scrO(1)_K}(Z)$ as in \cite[p. 964]{BGS}. We have:
		\begin{equation}\label{h_F_h_BGS} 
						h_{BGS}(Z)=h_{Fal}(Z)-[K:\Q]\sigma_{p}\deg_K(Z),
		\end{equation}
		where $\sigma_{p}$ is the Stoll number (see, for example, \cite[p. 922]{BGS}).
	\end{proposition}

	The arithmetic B\'ezout theorem \cite[Theorem 4.2.3]{BGS}
	implies the following:
		\begin{proposition}\label{prop:BezoutXY}
			Let $Y\in Z_{N-1}(\bP^N_{K})$
			and $Z\in Z_1(\bP^N_{K})$
			be two cycles intersecting properly in $\bP^N_{K}$.
			We have:
			\begin{align}\label{BezoutXY}
				\begin{split}
				h_{BGS}(Y.Z)\leq {}& \deg_K(Z)h_{BGS}(Y)+h_{BGS}(Z)\deg_K(Y)\\
				{} &+[K:\Q]a(N,N,2)\deg_K(Y)\deg_K(Z)
				\end{split}
			\end{align}
			where $a(N,N,2)$ is the constant defined in \cite[p. 971]{BGS}.
		\end{proposition}
		
		To prove Proposition \ref{prop:BezoutXY},
		note the following:
		$$h_{BGS}(Y.Z):=h_{BGS}(\overline{Y.Z})\leq h_{BGS}(\bar{Y}.\bar{Z})$$
		because the closure $\overline{Y.Z}$
		of $Y.Z$
		is contained in $\bar{Y}.\bar{Z}$. Then we bound $h_{BGS}(\bar{Y}.\bar{Z})$
		from above by the right hand side of (\ref{BezoutXY})
		thanks to \cite[Theorem 4.2.3]{BGS}.

		Let $H'$ denote the hyperplane of $\bP^N_K$ whose restriction to $X$ is $H$.
		Define $V'=\displaystyle\frac{V}{\sum_{i=1}^l m_i}$
		as a pure cycle (with rational coefficients) 
		in $\bP^N_F$. By the classical B\'ezout's
		theorem, we have:
		\begin{align}\label{1stdeg}
		\begin{split}	\deg_K(\pi_{*}V')={}&[F:K]\deg_F(V')=\frac{[F:K]\deg_F(V)}{\deg_{\bar{K}}(H'_{\bar{K}}.V_{\bar{K}})}\\
			={}&\frac{[F:K]\deg_F(V)}{\deg_F(H'_F.V)}=\frac{[F:K]}{\deg_F(H'_F)}=[F:K].
		\end{split}
		\end{align}
		\begin{equation}\label{2nddeg}
			\deg_K(H'.\pi_*V')=[F:K].
		\end{equation}

		Apply Proposition \ref{prop:BezoutXY} for the cycles $H'$
		and $\pi_*V'$ together with (\ref{1stdeg})
		and (\ref{2nddeg}), we have:
		\begin{equation}\label{ineq:1st_app}
			h_{BGS}(H'.\pi_*V')\leq [F:K]h_{BGS}(H')+h_{BGS}(\pi_*V')+[F:\Q]a(N,N,2).
		\end{equation}
		
		By using Proposition \ref{prop:h_F_h_BGS}, (\ref{1stdeg}),
		(\ref{2nddeg}) and the fact that $\sigma_0=0$, 
		we can replace $h_{BGS}$
		by $h_{Fal}$ in (\ref{ineq:1st_app}) to get:
		\begin{align}\label{ineq:h_F_only}
			\begin{split}
			h_{Fal}(H'.\pi_*V')\leq{}& [F:K](h_{Fal}(H')-[K:\Q]\sigma_{N-1})
			+h_{Fal}(\pi_*V')\\
			{}&-[F:\Q]\sigma_{1}
			+[F:\Q]a(N,N,2).
			\end{split}
		\end{align}
		Therefore
		\begin{equation}\label{ineq:c_10}
			h_{Fal}(H'.\pi_*V')\leq [F:K]h_{Fal}(H')+h_{Fal}(\pi_*V')+[F:\Q]c_{10},
		\end{equation} 	
		where $c_{10}=a(N,N,2)-\sigma_1-\sigma_{N-1}$ is an explicit
		constant depending only on $N$.
		
		Dividing both sides of (\ref{ineq:c_10}) by $[F:\Q]$, we have:
		\begin{equation}\label{dividing_F:Q}
		\frac{h_{Fal}(H'.\pi_*V')}{[F:\Q]}\leq \frac{h_{Fal}(H')}{[K:\Q]}+
		\frac{h_{Fal}(V')}{[K:\Q]}+c_{10}
		\end{equation} 
		
		From $H'_{\bar{K}}.V'_{\bar{K}}=\displaystyle\frac{\sum_{i=1}^l 
		m_iP_i}{\sum_{i=1}^l m_i}$, we have:
		\begin{equation}\label{eq:simplify}
		\frac{h_{Fal}(H'.\pi_*V')}{[F:K]}=\frac{\sum_{i=1}^l m_ih_{Fal}(P_i)}{\sum_{i=1}^l m_i}
		\end{equation}
		 
		Recall that $h$ denote the absolute Weil height on $\bP^N(\bar{K})$
		(see the paragraph right after Theorem \ref{thm:H_cap_V}). Note that
		$h_{Fal}$ on $\bP^N(\bar{K})$ is also a choice of a Weil height (relative 
		over $K$) corresponding the canonical line bundle $\scrO(1)$. Hence there 
		exists a constant $c_{11}$ such that:
		\begin{equation}\label{ineq:c_11}
		|h(P)-\frac{h_{Fal}(P)}{[K:\Q]}|\leq c_{11}\ \forall P\in \bP^N(\bar{K}).
		\end{equation}
		
		From (\ref{dividing_F:Q}), (\ref{eq:simplify}) and (\ref{ineq:c_11}),
		we have:
		$$\frac{\sum_{i=1}^l m_ih(P_i)}{\sum_{i-1}^l m_i}\leq \frac{h_{Fal}(V')}{[K:\Q]}+\frac{h_{Fal}(H')}{[K:\Q]}+c_{10}+c_{11}.$$
		To finish the proof of Theorem \ref{thm:H_cap_V}, it remains to show that 
		$\displaystyle\frac{h_{Fal}(V')}{[K:\Q]}$
		is bounded independently of $V$. We will use the canonical height 
		$h_{\varphi,L}$ constructed by Zhang \cite{Zhang95} and generalized by
		Kawaguchi \cite{Kawa06}. We have the following special case 
		of their construction:
		
		\begin{proposition}\label{prop:ZhKa}
			There is a height function $h_{\varphi,L}$ on
			effective cycles in $Z_1(X_{\bar{K}})$ satisfying the following
			properties:
			\begin{itemize}
				\item [(a)] If $Z$ is a preperiodic curve in $X_{\bar{K}}$ then
				$h_{\varphi,L}(Z)=0$.
				\item [(b)] There exists a constant $c_{12}$ such that for every
				curve $Z$ in $X_{\bar{K}}$, we have:
				$$|h_{\varphi,L}(Z)-\frac{h_{Fal}(Z)}{2[K:\Q]\deg_{\bar{K}}(Z)}|< c_{12}$$
			\end{itemize}
		\end{proposition}
		 	
		Part (a) follows from \cite[Theorem 2.4]{Zhang95}, and part (b) follows from
		\cite[Theorem 2.3.1]{Kawa06}. The preperiodicity of $V$ together with Proposition 
		\ref{prop:ZhKa} yield:
		$$\frac{h_{Fal}(V')}{[K:\Q]}=\frac{h_{Fal}(V)}{[K:\Q]\deg_{\bar{K}}(V)}<2c_{12}$$
		which finishes the proof of Theorem \ref{thm:H_cap_V}.
	
	\subsection{Split Polynomial Maps Associated to Disintegrated Polynomials}
	We briefly explain why Theorem 
	\ref{dynBMZ_1} and Theorem \ref{thm:C_cap_V}
	remain valid for the dynamics of maps of the form 
	$\Phi=(f_1,\ldots,f_n):\ (\bP^1_K)^n\longrightarrow (\bP^1_K)^n$,
	where $f_1,f_2,\ldots,f_n$ are disintegrated polynomials of degrees at least 2.
	This more general case can be easily reduced to the case of
	diagonally split polynomials maps $\varphi=(f,\ldots,f)$ considered throughout
	the paper.
	
	\begin{theorem}\label{thm:finalsplit}
	Let $n\geq 2$, and let $f_1,\ldots,f_n\in K[X]$  be disintegrated polynomials 
	of degrees at least 2. Then Theorem \ref{dynBMZ_1}
	and Theorem \ref{thm:C_cap_V}
	still hold for the dynamics of the split polynomial map
	$\Phi=(f_1,\ldots,f_n)$.
	\end{theorem}
	
	In fact, Medvedev and Scanlon (see Proposition 2.21 and 
	Fact 2.25 in \cite{MedSca}) prove that every irreducible $\Phi$-preperiodic
	hypersurface of $(\bP^1_K)^n$
	has the form $\pi_{ij}^{-1}(Z)$
	where $1\leq i<j\leq n$, $\pi_{ij}$
	is the projection onto the $(i,j)$-factor
	$(\bP^1)^2$ and $Z$ is an $(f_i,f_j)$-preperiodic
	curve in $(\bP^1_K)^2$. Therefore we can reduce to the case $n=2$.
	If every periodic curve of $(\bP^1)^2$ under $(f_1,f_2)$
	has the form $\zeta\times \bP^1$
	or $\bP^1\times \zeta$ then we are done. If there is a preperiodic curve
	that does not have such forms, by \cite[Proposition 2.34]{MedSca} there exist 
	polynomials $p_1$, $p_2$
	and $q$ such that $f_1\circ p_1=p_1\circ q$, and 
	$f_2\circ p_2=p_2\circ q$. In other words,
	we have the commutative diagram:
	$$\begin{diagram}
				 		\node{(\bP^1)^2}\arrow{s,l}{(p_1,p_2)}\arrow{e,t}{(q,q)}
				 		\node{(\bP^1)^2}\arrow{s,r}{(p_1,p_2)}\\
				 		\node{(\bP^1)^2}\arrow{e,b}{\Phi}\node{(\bP^1)^n}
				 \end{diagram}$$
	
	For every $\Phi$-preperiodic curve $V$ in $(\bP^1)^2$,
	we have that every irreducible component of $(p_1,p_2)^{-1}(V)$
	is $(q,q)$-preperiodic. More over, if $V$ is $\Phi$-periodic,
	at least one irreducible component of $(p_1,p_2)^{-1}(V)$
	is $(q,q)$-periodic. Furthermore, it is a consequence of Ritt's theory of polynomial
	decomposition that if $f_1\circ p_1=p_1\circ q$, and $f_1$ is disintegrated
	then $q$ is also disintegrated. Hence we can reduce to the case of diagonally
	split polynomial maps treated earlier.

	\bibliographystyle{amsalpha}
	\bibliography{Split_v16} 	

\end{document}